\documentclass[12pt]{amsart}
\usepackage{color,graphicx,enumerate,wrapfig,amssymb,mathtools}
\usepackage{pxfonts}
\usepackage{subfigure}
\usepackage{appendix}
\usepackage{eucal}

\newcommand{\ddiv}{{\mbox{div}}}

\newcommand{\R}{{\mathbb R}}

\newcommand{\cL}{{\mathcal L}}
\newcommand{\cS}{{\mathcal S}}

\newcommand{\cA}{{\mathcal A}}

\newcommand{\cD}{{\mathcal D}}

\newcommand{\cH}{{\mathcal H}}
\newcommand{\cN}{{\mathcal N}}

\newcommand{\e}{\varepsilon}

\newcommand{\spt}{\operatorname{spt}}
\newcommand{\diam}{\operatorname{diam}}
\newcommand{\dist}{\operatorname{dist}}

\newcommand{\p}{\partial}

\newcommand{\ra}{\rightarrow}

\newcommand{\norm}[1]{\left\Arrowvert {#1}\right\Arrowvert}

\theoremstyle{plain}
\newtheorem{theorem}{Theorem}[section]

\newtheorem{corollary}[theorem]{Corollary}

\newtheorem{lemma}[theorem]{Lemma}
\newtheorem{proposition}[theorem]{Proposition}

\theoremstyle{definition}
\newtheorem{definition}[theorem]{Definition}

\theoremstyle{remark}

\newtheorem{notation}[theorem]{Notation}
\newtheorem{remark}[theorem]{Remark}

\numberwithin{equation}{section}

\title [A free boundary ... conductivity]{An elliptic  free boundary arising from the jump of conductivity}
\author{Sunghan Kim}
\address{Seoul National University, Seoul, 151-747, Korea
}
\email{sunghan290@snu.ac.kr}

\author{Ki-ahm Lee}
\address{Seoul National University, Seoul, 151-747, Korea
\& Korea Institute for Advanced Study, Seoul,130-722, Korea}
\email{kiahm@snu.ac.kr}
\author[Henrik Shahgholian ]{Henrik Shahgholian}
\address{Department of Mathematics, Royal Institute of Technology,
  100~44  Stockholm, Sweden}
\email{henriksh@kth.se}
\thanks{S. Kim has been supported by NRF (National Research Foundation of Korea) Grant funded by the Korean Government (NRF-2014-Fostering Core Leaders of the Future Basic Science Program). K. Lee was supported by the National Research Foundation of Korea (NRF) grant funded by the Korea government (MSIP) (No. NRF-2015R1A4A1041675). K. Lee also holds a joint appointment with the Research Institute of Mathematics of Seoul National University. H. Shahgholian has been supported in part by Swedish Research Council. 
This project is part of an STINT (Sweden)-NRF (Korea) research cooperation  program. \\
\indent The authors are grateful to referees' comments that led to a substantial improvement of 
the first submitted draft.  John Andersson is acknowledged for invaluable input, specially for clarifications of several crucial  aspects in his work with coauthor \cite{AM}
}

\begin{document}
    
\begin{abstract}
In this paper we consider a quasilinear elliptic PDE, $\ddiv (A(x,u) \nabla u) =0$,  
where  the  underlying physical problem  gives rise to a jump for the conductivity $A(x,u)$, across a level surface for $u$. Our analysis concerns Lipschitz  regularity for the solution $u$, and the regularity of the level surfaces, where $A(x,u)$ has a jump and the solution $u$ does not degenerate. 

In proving Lipschitz regularity of solutions, we introduce a new and unexpected type of ACF-monotonicity formula with two different operators, that might be of independent interest, and surely can be applied in other related situations. The proof of the monotonicity formula is done through careful computations, and  (as a byproduct) a slight generalization to a specific type of variable matrix-valued conductivity  is presented.

 \end{abstract}

\maketitle

\tableofcontents

%
%

\section{Introduction}\label{section:intro}

\subsection{The model equation}
In this paper, we consider weak solutions $u$ of 
\begin{equation}\tag{$L$}\label{eq:main}
\ddiv(A(x,u)\nabla u)=0\quad\text{in }\Omega ,
\end{equation}
where $\Omega$ is a bounded domain in $\R^n$ ($n\geq 2$) and $A:\R^n\times\R\ra\R$ is defined by
\begin{equation}\label{eq:diffusion}
A(x,s):=a_-(x)+(a_+(x)-a_-(x))H(s)=a_-(x)(1-H(s))+a_+(x)H(s),
\end{equation}
where $H$ is heaviside function and $a_\pm\in C(\R^n)$ satisfying the following structure conditions:
\begin{itemize}
\item there is a $\lambda>0$ such that 
\begin{equation}\label{eq:unif ellip}
\lambda\leq\min\{a_-(x),a_+(x)\},\quad\max\{a_-(x),a_+(x)\}\leq\frac{1}{\lambda},\quad\forall x\in\R^n;
\end{equation}
\item there is a modulus of continuity $\omega$ such that
\begin{equation}\label{eq:cont}
\max\{|a_-(x)-a_-(y)|,|a_+(x)-a_+(y)|\}\leq\omega(|x-y|),\quad\forall x,y\in\R^n.
\end{equation}
\end{itemize}
We call $u$ a weak solution of \eqref{eq:main}, if $u\in W_{loc}^{1,2}(\Omega)\cap L^2(\Omega)$ satisfies
\begin{equation*}
\int  A(x,u)\nabla u\nabla\phi =0,\quad\forall \phi\in W_0^{1,2}(\Omega).
\end{equation*}

\subsection{Applications}

Heat or electric conduction through certain  materials  are usually,  due to 
complex properties of the materials, very hard to model. 
 For instance, mathematical modeling of  composites,  consisting of materials with  different conductivity properties, is one such problem which has been subject for intense (mathematical) studies. Applications of such models can be found in problems related to transmissions \cite{B},  and inverse and discontinuous  conductivity \cite{AI}, as well as other related problems \cite{BC}, where there is a jump in the conductivity.

The conductivity problem becomes substantially complicated when the materials also undergo a phase transition, which can cause  abrupt changes in the conductivity, due to a threshold of the heat or electric current.  The discontinuity in the conduction that  arise from a structural phase change in crystalline materials  has been considered in applied literature specially in relation to transport in solids, such as nanowires.  Relevant discussions in applied literature can be found in   \cite{BDHS}, \cite{WLSXMZ}. The mathematical problem of jump in conductivity across level surfaces
have also been considered recently in  \cite{ACS},  \cite{AM} and \cite{AT}. 

 The simplest model of a material-dependent 
 conductivity can be written as $A(x,u)\nabla u$, where now $A(x,u)$ has a discontinuity in the $u$-variable, which represents the  heat, or electric charge. In this paper, we have chosen a simple elliptic model, where the model equation is written as $A(x,u) = a_+(x)\chi_{\{u> 0\}} +   a_-(x) \chi_{\{u\leq 0\}}$.

It should be remarked that heat conduction in certain materials, which also undergo a phase transition, may  cause a  change in conductivity only at the phase, in terms of latent heat on the boundary between two-phases. Typical examples are melting ice or flame propagation. The mathematical study of such problems, usually entitled Bernoulli free boundaries, have been carried out in a large scale in the last few decades, e.g., \cite{AC}, \cite{ACF}, \cite{CS} and the references therein. Our problem, although qualitatively different, carries features  reminiscent of that of   Bernoulli type problem for  the  latent heat.

\subsection{Methodology and Approach}

As appearing in \eqref{eq:diffusion}, the jump of $A(x,u)$ in the second argument across $\{u=0\}$ naturally induces a free boundary condition, which can be formally represented as 
\begin{equation*}
a_+(x)u_\nu^+=a_-(x)u_\nu^-,
\end{equation*}
where $\nu$ points towards level sets where $u$ increases. 

Assuming the continuity \eqref{eq:cont} of our coefficients, it is not hard to see that our limiting equation after blowing up the solution at a point $z$ is of the form 
\begin{equation*}
\ddiv((a_-(z)+(a_+(z)-a_-(z))H(v))\nabla v)=0,
\end{equation*}
so that the function $w$, defined by
\begin{equation*}
w(x):=a_+(z)v^+(x)-a_-(z)v^-(x),
\end{equation*} 
becomes harmonic.\footnote{For this reason, it seems natural  to call our limit profile a ``broken'' harmonic function.} Hence, the heart of our analysis lies in seeking a way to bring back the strong properties of our limit profiles to the original solutions before blowup. 

Our first result concerns with the Lipschitz regularity of weak solutions to \eqref{eq:main}. In this direction, we develop a new monotonicity formula of the ACF type, involving two different operators. A key issue in deriving an ACF type monotonicity formula is to make sure that the concentrated energy on each component ($I(r,u_+)$ and $I(r,u_-)$ in Appendix \ref{section:apndx}) can be compared on the same sphere, so that the eigenvalue inequality and the Friedland-Hayman inequality can be applied to both components simultaneously. In this direction, one possible choice of two different operators is that the associated coordinate systems can be simultaneously rotated into concentric balls; see Theorem \ref{theorem:acf}. One may notice that the two distinct operators associated with \eqref{eq:main} certainly enjoy this property. As the ACF monotonicity being adapted to linear scaling, we may get rid of any possibility that a sublinearly scaled blowup limit at a free boundary point ends up with a nontrivial function, proving the Lipschitz regularity of our solutions. 

Our analysis on the regularity of our free boundary is conducted by making an exclusive use of measure $\mu=\ddiv(a_+(x)\nabla u^+)$, which turns out to carry essential information on our free boundary points. Our main interest is in twofold; first, smoothness of our free boundary around nondegenerate points, see Definition \ref{definition:N-S}, and second, characterization of $\mu$ in terms of the $(n-1)$-dimensional Hausdorff measure. We first make a crude observation that rapidly vanishing points (Lemma \ref{lemma:D}) carries null $\mu$-measure, which allows us to focus on vanishing points with finite orders. Essentially, finitely vanishing points admit doubling conditions, which stablize the blowup sequence and after all give us a nontrivial limit. With this stability at hand, we may follow the arguments in \cite{AM} to deduce that $\mu$-almost every free boundary point has a two plane solution as its blowup limit, consequently yielding a flatness condition. 

Assuming H\"{o}lder regularity on each part of our coefficients, we invoke a linearization technique to improve the flatness condition in an inductive way, where the technique itself is a well established theory that by now can be considered as classical. Through an iteration argument, we deduce that our free boundary around a flat point can be trapped locally uniformly by two $C^{1,\alpha}$ graphs, and thus proving its $C^{1,\alpha}$ regularity. Equipped with the $C^{1,\alpha}$ regularity theory, we are also able to rule out degenerate points (Definition \ref{definition:N-S}) to have any flatness condition, which well matches our intuition, and finally conclude that the support of $\mu$ consists of nondegenerate points, from which we deduce its $\sigma$-finiteness of $(n-1)$-dimensional Hausdorff measure. 

\subsection{Relation to existing literatures}
The model equation with varying coefficients have not been treated earlier, even when the coefficient on each component is assumed to be smooth. To the authors' best knowledge, the only existing literature so far is that of [AM], which treats the conductivity jump with constant matrices. The authors of \cite{AM} do not prove  Lipschitz regularity for solutions, as this is not true in general. For the parabolic case, Caffarelli and Stefanelli proved in \cite{Ca-St} that solutions to similar problems, with conductivity jump, cannot be Lipschitz regular in space. In a recent result \cite{CDS}, the Lipschitz regularity is proved only in two dimensional case for a similar problem, using geometric methods. According to the authors of \cite{CDS}, one can construct in higher dimensions non-Lipschitz solutions to a Bernoulli type free boundary problem with two different operators and a nonlinear Bernoulli boundary gradient condition. Thus, it should be remarked that the problem under our consideration introduces substantial difficulty, not undertaken earlier in the literatures, and we believe that our results in this paper are original.

To the best of our knowledge, our result on the Lipschitz regularity of solutions is completely new. As mentioned above, even in the scalar case and for constant-matrix valued conductivity there are no such results in the literature, since this in general is not true. Besides, the way we derive the Lipschitz regularity has its independent value in the regularity theory for elliptic PDEs, since we invoke a free boundary technique to detect the optimal regularity of weak solutions to PDEs whose coefficients have specified discontinuity. It should also be stressed that as mentioned earlier, our newly developed monotonicity formula can be extended even for certain matrix valued conductivities of specific types, and thus will probably have an impact in solving other related problems where there are varying operators on each phase of a free boundary problem with a phase transition. 

It is noteworthy that  finding counterexamples to Lipschitz regularity for the matrix coefficient case 
seems a challenging problem. Also finding optimal conditions for Lipschitz regularity of the matrix case remains unanswered. It seems however not obvious whether Lipschitz regularity is a necessity for smoothness of the free boundary or not.

\subsection{Organization of the paper} We prove the existence of weak solutions of \eqref{eq:main} in Section \ref{section:existence} and give a weak formulation of our free boundary condition. Section \ref{section:regularity} is devoted to the optimal regularity of our solutions (Theorem \ref{theorem:opt reg}). In Section \ref{section:prelim} we introduce the measure $\mu=L_+u^+$ and derive some of its basic properties. In section \ref{section:hausdorff} we observe further properties of our free boundary by means of $\mu$ and prepare for Section \ref{section:higher reg}. Section \ref{section:higher reg} is devoted to prove that our free boundary is a $C^{1,\alpha}$ graph in a neighborhood of $\mu$ almost every point (Theorem \ref{theorem:C1a-fb}), and that the support of $\mu$ has $\sigma$-finite $(n-1)$ dimensional Hausdorff measure (Theorem \ref{theorem:Hn-1}). In Section \ref{section:matrix} we extend all the prescribed results for certain class of varying matrix valued conductivity. Finally in Appendix \ref{section:apndx} we give a proof for the newly developed ACF monotonicity formula with two different operators (Theorem \ref{theorem:acf}).

Throughout this paper, we say a constant to be universal if it depends only on $n,\lambda$ and $\omega$, those appearing in \eqref{eq:unif ellip} and \eqref{eq:cont}. Moreover, any constant will be considered positive universal, unless otherwise stated. Also given a function $u$, $u^+:=\max\{u,0\}$ and $u^-:=-\min\{u,0\}$, while by the subscript notation (e.g., $u_\pm$) we just denote two different quantities.

%
%

\section{Existence and Free Boundary Condition}\label{section:existence}

Let us begin with the existence theory of our problem \eqref{eq:main}. 

\begin{proposition}\label{proposition:exist} Given $g\in W^{1,2}(\Omega)$, there exists a weak solution $u$ of \eqref{eq:main} in $\Omega$ such that $u-g\in W_0^{1,2}(\Omega)$.
\end{proposition}

\begin{proof} Let $\psi:\R\ra\R$ be defined by
\begin{equation*}
\psi(t):=\begin{cases}
1&\text{if }t>1,\\
t&\text{if }0\leq t\leq 1,\\
0&\text{if }t<0,
\end{cases}
\end{equation*}
and define $\psi_\e$ and $\Psi_\e$ for each $\e>0$ respectively by
\begin{equation*}
\psi_\e(t):=\psi\left(\frac{t}{\e}\right)\quad\text{and}\quad\Psi_\e(t):=\int_{-\infty}^t\psi_\e(s)ds.
\end{equation*}

Set $A_\e(x,s):=a_-(x)+(a_+(x)-a_-(x))\psi_\e(s)$. Since $x\mapsto A_\e(x,s)$ is measurable on $\Omega$ and $s\mapsto A(x,s)$ is Lipschitz on $\R$, there exists a weak solution $u_\e\in W^{1,2}(B_1)$ for
\begin{equation*}
\ddiv(A_\e(x,u_\e)\nabla u_\e)=0\quad\text{in }\Omega,
\end{equation*}
with $u_\e-g\in W_0^{1,2}(\Omega)$. It also follows from the uniform ellipticity of $A_\e(x,s)$, with ellipticity constants being $\lambda$ and $\frac{1}{\lambda}$ (see \eqref{eq:unif ellip}) independent of $\e>0$, we will have a uniform estimate
\begin{equation}\label{eq:unif est-ue}
\norm{u_\e}_{W^{1,2}(\Omega)}\leq C\norm{g}_{W^{1,2}(\Omega)},
\end{equation}
for some universal $C>0$.

From this fact we can derive another uniform estimate for $\Psi_\e(u_\e)$ in the norm $\norm{\cdot}_{W^{1,2}(\Omega)}$. Due to the fact that $\Psi_\e(s)\leq s$ for $s>0$ and $\Psi_\e(s)\equiv 0$ for $s\leq 0$, we obtain
\begin{equation*}
\int  \Psi_\e(u_\e)^2\leq \int  u_\e^2. 
\end{equation*}
Then we use the chain rule to derive
\begin{equation*}
\int |\nabla \Psi_\e(u_\e)|^2=\int |\psi_\e(u_\e)\nabla u_\e|^2\leq \int |\nabla u_\e|^2.
\end{equation*}
Therefore, 
\begin{equation}\label{eq:unif est-phie}
\norm{\Psi_\e(u_\e)}_{W^{1,2}(\Omega)}\leq \norm{u_\e}_{W^{1,2}(\Omega)}\leq C\norm{g}_{W^{1,2}(\Omega)}.
\end{equation}

Now we are in position to take limits. From \eqref{eq:unif est-ue} there is a subsequence $\e_k\searrow 0$ and $u\in W^{1,2}(\Omega)$ such that
\begin{equation}\label{eq:conv-ue}
\begin{split}
u_{\e_k}&\ra u\quad\text{strongly in }L^2(\Omega),\\
\nabla u_{\e_k}&\ra \nabla u\quad\text{weakly in }L^2(\Omega).
\end{split}
\end{equation}
Note that the strong $L^2$ convergence implies $u-g\in W_0^{1,2}(\Omega)$. 

Next we apply the estimate \eqref{eq:unif est-phie} along this subsequence. It allows us to extract a further subsequence, which we still denote by $\e_k$, and some $v\in W^{1,2}(\Omega)$ such that $\Psi_{\e_k}(u_{\e_k})\ra v$ strongly in $L^2(\Omega)$ and $\nabla \Psi_{\e_k}(u_{\e_k})\ra \nabla v$ weakly in $L^2(\Omega)$. However, since $\Psi_\e(s)\ra\max\{s,0\}$ uniformly in $\R$, we must have $v=u^+$. Also recall again that $\nabla\Psi_\e(u_\e)=\psi_\e(u_\e)\nabla u_\e$. Thus, rephrasing the convergence of $\{\Psi_{\e_k}(u_{\e_k})\}$, we arrive at
\begin{equation}\label{eq:conv-phie}
\begin{split}
\Psi_{\e_k}(u_{\e_k})&\ra u^+\quad\text{strongly in }L^2(\Omega),\\
\psi_\e(u_\e)\nabla u_\e&\ra \nabla u^+\quad\text{weakly in }L^2(\Omega).
\end{split}
\end{equation}

Now let us verify that the limit function $u$ is a weak solution for 
\begin{equation*}
\ddiv(A(x,u)\nabla u)=0\quad\text{in }\Omega.
\end{equation*}
Let $\phi\in C_c^\infty(\Omega)$ be a given. Then the weak convergence in \eqref{eq:conv-ue} and \eqref{eq:conv-phie} implies respectively that
\begin{equation*}
\lim_{k\ra\infty}\int  a_-\nabla u_{\e_k}\nabla\phi=\int  a_-\nabla u\nabla \phi,
\end{equation*}
and that
\begin{equation*}
\begin{split}
\lim_{k\ra\infty}\int  (a_+-a_-)\psi_{\e_k}(u_{\e_k})\nabla u_{\e_k}\nabla\phi&=\int (a_+-a_-)\nabla u^+\nabla\phi\\
&=\int (a_+-a_-)H(u)\nabla u\nabla\phi,
\end{split}
\end{equation*}
where we have used that $\nabla u^+=\nabla u$ a.e. on $\{u>0\}$ and $\nabla u^+=0$ a.e. on $\{u\leq 0\}$. Combining these two limits, we get
\begin{equation*}
\begin{split}
\int  A(x,u)\nabla u\nabla\phi&=\lim_{k\ra\infty}\int  A_{\e_k}(x,u_{\e_k})\nabla u_{\e_k}\nabla\phi =0,
\end{split}
\end{equation*}
as desired. 
\end{proof}

\begin{remark}\label{remark:exist} Note that we have not made any assumption on the regularity of $a_\pm$ in Proposition \ref{proposition:exist}. We also make another remark in the proof that the weak star convergence of $\psi_{\e_k}(u_{\e_k})\ra H(u)$ in $L_{loc}^\infty(\Omega)$ is in general not true. 
\end{remark}
 
\begin{notation}\label{notation:u} From now on $u$ is a nontrivial weak solution of \eqref{eq:main} in $\Omega$, unless otherwise stated. Also $\Omega^+(u):=\{u>0\}\cap\Omega$, $\Omega^-(u):=\{u<0\}\cap\Omega$ and $\Gamma(u):=\p\{u>0\}\cap\Omega$. Moreover, $L_\pm$ are the operators defined by $L_\pm u=\ddiv(a_\pm(x)\nabla u)$. We will also use the notation $A_z(s):=A(z,s)$ to emphasize that the $z$-argument is being fixed.
\end{notation} 

Although the free boundary condition is already dictated by the equation, it is by no means elementary to show that it holds at all free boundary points. Here we shall show a weak version of the free boundary condition, which holds for all free boundaries and not only nondegenerate ones. 

\begin{lemma}\label{lemma:weak-FB-condition} Suppose that $a_\pm$ are Dini continuous, and let $u$ be a weak solution for \eqref{eq:main} in $\Omega$. Assume further that sets $\{u>\e\}\cap\Omega$ and $\{u<-\e\}\cap\Omega$ have locally finite perimeter for each $\e>0$. Then 
\begin{equation}\label{eq:weak-FB-condition}
\lim_{\e\searrow 0}\int_{\p\{u>\e\}\cap\Omega}a_+|\nabla u|\eta = \lim_{\e\searrow 0}\int_{\p\{u<-\e\}\cap\Omega}a_-|\nabla u|\eta,
\end{equation}
for any $\eta\in C_c^\infty(\Omega)$.
\end{lemma}

\begin{proof} Since $u\in W^{1,2}(\Omega)$, $\nabla u=0$ a.e. on $\{u=0\}$. Then  
\begin{equation*}
\begin{split}
0&=\int_{\Omega} A(x,u)  \nabla u \nabla \eta\\
&=
\lim_{\e\searrow 0}\int_{\{u>\epsilon \}} a_+ \nabla u \nabla   \eta +
\lim_{\e\searrow 0}\int_{\{u < -\epsilon \}}  a_- \nabla u \nabla \eta.
\end{split}
\end{equation*}
From the Dini continuity of $a_+$, we know that $u\in C^1(\Omega^+(u))$. Moreover, from the assumption that $\{u>\e\}\cap\Omega$ has locally finite perimeter, we may adopt integration by part and deduce that 
\begin{equation*}
\int_{\{u>\e\}\cap\Omega} a_+\nabla u\nabla \eta = \int_{\p\{u>\e\}\cap\Omega}a_+\nabla u\cdot\nu_{\{u>\e\}\cap\Omega}\eta, 
\end{equation*}
where $\nu_{\{u>\e\}\cap\Omega}$ is the measure theoretic outward unit normal of $\p\{u>\e\}\cap\Omega$; for the definition of measure theoretic unit normals, see Section 5 of \cite{EG}. In particular, if $x$ is a point on $\p\{u>\e\}\cap\Omega$ such that $\nabla u(x)\neq 0$, then $\nu_{\{u>\e\}\cap\Omega}(x) = -\nabla u(x) /|\nabla u(x)|$. Therefore, one may rephrase the surface integral on the right hand side by 
\begin{equation*}
\int_{\p\{u>\e\}\cap\Omega}a_+\nabla u\cdot\nu_{\{u>\e\}\cap\Omega}\eta = -\int_{\p\{u>\e\}\cap\Omega} a_+|\nabla u|\eta.
\end{equation*}
By a similar reason, one may also derive that 
\begin{equation*}
\int_{\{u<-\e\}\cap\Omega} a_-\nabla u \nabla \eta = \int_{\p \{u<-\e\}\cap\Omega} a_-|\nabla u|\eta. 
\end{equation*}
Summing all up, we arrive at 
\begin{equation*}
\begin{split}
0= -\lim_{\e\searrow 0}\int_{\p\{u>\e\}\cap\Omega} a_+|\nabla u|\eta + \lim_{\e\searrow 0}\int_{\p \{u<-\e\}\cap\Omega} a_-|\nabla u|\eta,
\end{split}
\end{equation*} 
as desired.
\end{proof}

\begin{remark}\label{remark:weak-FB-condition} The assumption that $\{u>\e\}\cap\Omega$ and $\{u<-\e\}\cap\Omega$ has locally finite perimeters is met, under the circumstance that $a_+$ and $a_-$ are Lipschitz continuous; see Theorem 5.2.1 in \cite{HL}. It is not easy, however, to derive the same conclusion even with H\"{o}lder regular $a_+$ and $a_-$.
\end{remark}

\section{Regularity of Solutions}\label{section:regularity}

In this section, we prove interior regularity of our solutions to \eqref{eq:main}. We know from the De Giorgi theory that solutions of \eqref{eq:main} is locally H\"{o}lder continuous for some universal exponent $0<\alpha<1$, even when our coefficients $a_\pm$ are assumed to be bounded measurable. Our first observation is that as we assume the continuity assumption \eqref{eq:cont}, our solutions become locally H\"{o}lder continuous for any exponent arbitrarily close to $1$. 

\begin{lemma}\label{lemma:Ca} For each $\alpha\in(0,1)$ there exists some $C_\alpha>0$ such that for $D\Subset\Omega$ and $z\in\{u=0\}\cap D$,
\begin{equation*}
|u(x)|\leq C_\alpha\norm{u}_{L^\infty(D)}\frac{|x-z|^\alpha}{d^\alpha}\quad\text{in }D,
\end{equation*}
where $d=\dist(D,\p\Omega)$.
\end{lemma}

\begin{proof} By a scaling argument, it suffices to show it only for $\Omega=B_1$, $D=B_{1/2}$ and $u$ such that $\norm{u}_{L^\infty(D)}=1$.

By the De Giorgi theory, there is some universal $\bar{\alpha}\in(0,1)$ for which the statement of this proposition is true. Thus, we only need to consider $\alpha\in(\bar{\alpha},1)$. 

Suppose towards a contradiction that there exist solutions $u_j$ of \eqref{eq:main} in $B_1$, $x^j\in\Gamma\cap B_{1/2}$ and $r_j<1/4$ such that $r_j\searrow 0$ and
\begin{equation*}
\sup_{B_{r_j}(x^j)}|u_j|=jr_j^\alpha\quad\text{and}\quad\sup_{B_r(x^j)}|u_j|\leq jr^\alpha\quad\text{for all }r_j\leq r\leq 1/4.
\end{equation*}
Consider
\begin{equation*}
v_j(x):=\frac{u_j(r_jx+x^j)}{jr_j^\alpha},\quad x\in B_{1/4r_j}.
\end{equation*}
Then $v_j$ solves 
\begin{equation*}
\ddiv(A_j(x,v_j)\nabla v_j)=0\quad\text{in }B_{1/4r_j},
\end{equation*}
where $A_j(x,s):=A(r_jx+x^j,s)$, and that
\begin{equation*}
\sup_{B_1}|v_j|=1\quad\text{and}\quad\sup_{B_R}|v_j|\leq R^\alpha\quad\text{for all }1\leq R\leq 1/(4r_j).
\end{equation*}

Notice that $\lambda\leq A_j(x,s)\leq\frac{1}{\lambda}$ for all $(x,s)\in\R^n\times\R$. Thus the uniform supremum estimate above yields that there is some $v\in W_{loc}^{1,2}(\R^n)\cap C_{loc}^{\bar{\alpha}}(\R^n)$ such that
\begin{equation}\label{eq:conv-Ca}
\begin{split}
v_j&\ra v\quad\text{locally uniformly in }\R^n,\\
\nabla v_j&\ra\nabla v\quad\text{weakly in }L_{loc}^2(\R^n),
\end{split}
\end{equation}
along a subsequence, which we still denote by $j$, and that
\begin{equation}\label{eq:v-Ca}
\sup_{B_1}|v|=1\quad\text{and}\quad\sup_{B_R}|v|\leq R^\alpha,\quad\forall R\geq 1.
\end{equation}

Let us derive the equation which $v$ satisfies. Recall that $x^j\in B_{1/2}$, so there is $x^0\in\bar{B}_{1/2}$ to which $x^j$ is convergent along a subsequence. Let us again denote this convergent subsequence by $j$. Then the continuity assumption \eqref{eq:cont} on $a_\pm$\footnote{We make a remark here that only continuity is used.} and the weak convergence in \eqref{eq:conv-Ca} imply that
\begin{equation*}
\begin{split}
0&=\lim_{j\ra\infty}\int A_j(x,v_j)\nabla v_j\nabla\phi\\
&=\lim_{j\ra\infty}\left[\int a_+(r_jx+x^j)\nabla v_j^+\nabla\phi-\int a_-(r_jx+x^j)\nabla v_j^-\nabla\phi\right]\\
&=\int a_+(x^0)\nabla v^+\nabla\phi-\int a_-(x^0)\nabla v^-\nabla\phi\\
&=\int A_{x^0}(v)\nabla v\nabla\phi,
\end{split}
\end{equation*}
for any $\phi\in C_c^\infty(\R^n)$. Thus, $v$ is a weak solution to
\begin{equation}\label{eq:v}
\ddiv(A_{x^0}(v)\nabla v)=0\quad\text{in }\R^n.
\end{equation}

Now define 
\begin{equation*}
w:=a_+(x^0)v^+-a_-(x^0)v^-\quad\text{in }\R^n.
\end{equation*}
It immediately follows from \eqref{eq:v} and the definition of $A_{x^0}(v)$ that $w$ is harmonic in the entire space $\R^n$. From \eqref{eq:v-Ca}, we have 
\begin{equation*}
|w(x)|\leq\frac{1}{\lambda}|v(x)|\leq\frac{1}{\lambda}|x|^\alpha\quad\text{for }|x|\geq 1.
\end{equation*}
Hence, the Liouville theorem implies that $w$ is a constant function in $\R^n$. Now from the uniform convergence in \eqref{eq:conv-Ca}, we have $v(0)=\lim_{j\ra\infty}v_j(0)=0$. Thus, 
\begin{equation*}
w\equiv w(0)=0\quad\text{in }\R^n.
\end{equation*}
However, from the equality in \eqref{eq:v-Ca},
\begin{equation*}
\sup_{B_1}|w|=\sup_{B_1}(a_+(x^0)v^++a_-(x^0)v^-)\geq\lambda\sup_{B_1}|v|=\lambda>0,
\end{equation*}
a contradiction.
\end{proof}

By assigning higher regularity on $\omega$, we are able to improve Lemma \ref{lemma:Ca} to Lipschitz regularity.

\begin{lemma}\label{lemma:lip} Suppose that $a_+$ and $a_-$ satisfy \eqref{eq:cont} with a Dini continuous $\omega$.\footnote{A modulus of continuity $\omega$ is said to be Dini continuous, if $\int_0^1\frac{\omega(r)}{r}dr<\infty$.} There exists some $C>0$ such that for $D\Subset \Omega$ and $z\in\{u=0\}\cap D$, 
\begin{equation*}
|u(x)|\leq C\norm{u}_{L^\infty(D)}\frac{|x-z|}{d}\quad\text{in }D,
\end{equation*}
where $d=\dist(D,\p\Omega)$.
\end{lemma}

\begin{proof} As in the proof of Lemma \ref{lemma:Ca}, we simplify the setting and argue by contradiction. Suppose that there exist solutions $u_j$ of \eqref{eq:main} in $B_1$, $x^j\in\Gamma\cap B_{1/2}$ and $r_j<1/4$ such that $r_j\searrow 0$ and
\begin{equation*}
\sup_{B_{r_j}(x^j)}|u_j|=jr_j\quad\text{and}\quad\sup_{B_r(x^j)}|u_j|\leq jr\quad\text{for all }r_j\leq r\leq 1/4.
\end{equation*}
We know that $x^j\ra x^0$ for some $x^0\in\bar{B}_{1/2}$ up to a subsequence, and let us overwrite the indices of this subsequence by $j$. 

Consider
\begin{equation*}
v_j(x):=\frac{u_j(r_jx+x^j)}{jr_j},\quad x\in B_{1/4r_j}.
\end{equation*}
Then following the argument in the proof of Lemma \ref{lemma:Ca}, there is some $v\in W_{loc}^{1,2}(\R^n)\cap C_{loc}^\alpha(\R^n)$ with a universal $\alpha\in(0,1)$ such that \eqref{eq:conv-Ca} is true along a subsequence, which we still denote by $j$, and that
\begin{equation}\label{eq:v-lip}
\sup_{B_1}|v|=1. 
\end{equation}
Moreover, $v$ is a weak solution to \eqref{eq:v}. Since $\nabla v_j\ra\nabla v$ weakly in $L_{loc}^2(\R^n)$, 
\begin{equation*}
|x|^{\frac{2-n}{2}}\nabla v_j^\pm\ra |x|^{\frac{2-n}{2}}\nabla v^\pm\quad\text{weakly in }L_{loc}^2(\R^n).
\end{equation*}
As a result, for each $R\geq 1$, 
\begin{equation*}
\int_{B_R}\frac{|\nabla v^\pm|^2}{|x|^{n-2}}\leq\varliminf_{j\ra\infty}\int_{B_R}\frac{|\nabla v_j^\pm|^2}{|x|^{n-2}}.
\end{equation*}
Now apply to $u_j$ the ACF estimate, that is, Proposition \ref{proposition:acf-est}, to arrive at
\begin{equation*}
\Phi(R,v_j^+,v_j^-)=\frac{1}{j^4}\Phi(Rr_j,x^j,u_j^+,u_j^-)\leq \frac{256C}{j^4},
\end{equation*}
for all $j$ such that $r_j\leq\frac{r_0}{4R}$, where $r_0$ and $C$ are chosen as in Proposition \ref{proposition:acf-est}; see \eqref{eq:acf} for the definition of $\Phi$. Thus,
\begin{equation*}
\begin{split}
\Phi(R,v^+,v^-)&\leq\varliminf_{j\ra\infty}\Phi(R,v_j^+,v_j^-)\leq\varliminf_{j\ra\infty}\frac{256C}{j^4}=0,
\end{split}
\end{equation*}
for each $R\geq 1$. This implies that either $\nabla v^+$ or $\nabla v^-$ vanishes a.e. in $\R^n$.

Without losing any generality we may assume $\nabla v^+=0$ a.e. in $\R^n$. Then the continuity of $v^+$ implies that $v^+\equiv 0$ in $\R^n$. In view of \eqref{eq:v}, this implies that $v^-$ is harmonic in the entire space $\R^n$. Then the maximum principle implies that $v^-\equiv 0$ in $\R^n$. Finally, we obtain $v\equiv 0$ in $\R^n$, which violates \eqref{eq:v-lip}. 
\end{proof}

Now we are ready to state our first main result.

\begin{theorem}\label{theorem:opt reg}
Let $a_+$ and $a_-$ satisfy \eqref{eq:cont} with a Dini continuous $\omega$ and let $u$ be a weak solution to \eqref{eq:main} in $\Omega$. Then $u\in W_{loc}^{1,\infty}(\Omega)$ and for any $D\Subset \Omega$, 
\begin{equation}\label{eq:opt reg}
\norm{\nabla u}_{L^\infty(D)}\leq\frac{C\norm{u}_{L^2(\Omega)}}{d^{\frac{n}{2}+1}},
\end{equation}
where $C$ depends only on $n$, $\lambda$ and $\omega$, and $d=\dist(D,\p\Omega)$. 
\end{theorem}

\begin{proof} Fix $D\Subset\Omega$ with $d=\dist(D,\p\Omega)>0$ and choose $\tilde{D}$ such that $D\Subset\tilde{D}\Subset\Omega$ and $\dist(D,\p\tilde{D})=\dist(\tilde{D},\p\Omega)=d/2$. Write $M=\norm{u}_{L^2(\Omega)}$. By the local boundedness of weak solutions (Theorem 8.17 in \cite{GT}), we know that $K:=\norm{u}_{L^\infty(\tilde{D})}\leq cM/d^{n/2}$ for a constant $c$ depending only on $n$ and $\lambda$. 

Choose any $z\in D$ such that $u(z)\neq 0$ and set $r=\dist(z,\{u=0\}\cap D)$. If $r\geq d/2$, then the standard regularity theory yields that $|\nabla u(z)|\leq C_0K/d$. 

On the other hand, if $r<d/2$, then we consider the scaled function $v(x):=u(rx+z)/r$ in $B_1$. Suppose that $u(z)>0$. Then $v$ is a nonnegative weak solution for $\ddiv(a_+(rx+z)\nabla v)=0$ in $B_1$. Since $a_+$ is Dini continuous, the standard regularity theory yields that $v\in C^1(B_1)$. Then the Harnack inequality implies that $|\nabla v(0)|\leq C_1 v(0)$. By means of Lemma \ref{lemma:lip}, we have $v(0)=u(z)/r\leq C_2K/d$. 

A similar argument applies to the case $u(z)<0$. Using $K\leq cM/d^{\frac{n}{2}+1}$, we finish the proof. 
\end{proof}

%
%

\section{Preliminaries properties of  the Free Boundary}\label{section:prelim}
It is noteworthy that the zero set of $u$ may have vanishing points of infinite order, or even contain an open subset while keeping the solution from being trivial. This is true for general uniformly elliptic partial differential equations with only H\"{o}lder continuous coefficients; e.g., see \cite{M}. For this reason, we may expect the same situation to occur in our case; note that if $v$ solves $\ddiv(a(x)\nabla v)=0$, then $v^+/\beta_+ - v^-/\beta_-$ for some distinct $\beta_\pm>0$ solves \eqref{eq:main} with $a_\pm := a(x)/\beta_\pm$, and has the same zero level set with $v$. Nevertheless, we still the following lemma due to the Harnack inequality. 

\begin{lemma}\label{lemma:gamma} Let $u$ be a weak solution of \eqref{eq:main} in $\Omega$. Then $\Gamma(u)=\p\{u<0\}\cap \Omega$. 
\end{lemma}

\begin{proof} Let $z\in\Gamma(u)$ and let us show that $z\in\p\{u<0\}\cap \Omega$ as well. Take any open ball $B$ centered at $z$ and the concentric open ball $B'$ with half the radius. By the choice of $z$, we have $\sup_Uu>0$ and $\inf_Uu\leq 0$ for both $U=B$ and $B'$. Consider $v=u-\inf_Bu$ which is nonnegative on $B$. Since $v$ satisfies $\ddiv(A(x,u)\nabla v)=0$ in $B$, it follows from the Harnack inequality for bounded measurable coefficients that $\sup_{B'}v\leq C\inf_{B'}v$ with some universal $C>1$, whence 
\begin{equation}\label{eq:harnack}
0<\sup_{B'}u\leq C\inf_{B'}u-(C-1)\inf_{B}u\leq -(C-1)\inf_{B}u.
\end{equation}
Therefore, we can choose $y\in B$ such that $u(y)<0$. Since $B$ was chosen to be an arbitrary open ball centered at $z$, we get that $z\in\p\{u<0\}\cap \Omega$. It implies that $\Gamma(u)\subset\p\{u<0\}\cap \Omega$. By a similar argument we obtain the reverse inclusion, so we conclude that $\Gamma(u)=\p\{u<0\}\cap \Omega$.
\end{proof}

We are interested in the measure $\mu$ defined by $d\mu=L_+u^+$; i.e.,
\begin{equation*}
\int\phi d\mu:=\int a_+\nabla u^+\nabla\phi,\quad\forall \phi\in C_c^\infty(\Omega).
\end{equation*}

Let us make a basic observation on the measure $\mu$. 

\begin{lemma}\label{lemma:pos-msr} The measure $\mu$ is a positive Radon measure with $\spt(\mu)\subset\Gamma(u)$. Moreover, $\mu$ is locally finite and satisfies for any $B_{2r}(z)\subset\Omega$, 
\begin{equation}\label{eq:pos-msr}
\frac{\mu(B_r(z))}{r^{\frac{n}{2}-2}}\leq C\norm{u}_{L^2(B_{2r}(z))},
\end{equation}
where $C>0$ is a constant depending only on $n$ and $\lambda$. 
\end{lemma}

\begin{proof} If $\phi\in C_c^\infty(\Omega^+(u))$, then since $\nabla u^+ = \nabla u$ a.e. in $\Omega^+(u)$, we have 
\begin{equation*}
\int a_+\nabla u^+\nabla \phi = \int a_+\nabla u\nabla \phi = 0,
\end{equation*}
from which we deduce that $\spt(\mu) \subset\Gamma(u)$. 

Next we prove that $\mu$ is a positive measure. Suppose that $\phi\in C_c^\infty(\Omega)$ is nonnegative on $\Omega$. Let us choose $\psi_\e$ for each $\e>0$ in such a way that
\begin{equation*}
0\leq \psi_\e\leq 1,\quad \psi_\e'\geq 0,\quad \psi_\e(t) = 0\text{ for }t\leq \e/2\text{ and }\psi_\e(t) =1\text{ for }t\geq \e.
\end{equation*}
Then since $\psi_\e(u)\phi\in W_0^{1,2}(\Omega^+(u))$, we deduce from the weak formulation of \eqref{eq:main} that 
\begin{equation*}
0 = \int a_+\nabla u\nabla (\psi_\e(u)\phi) = \int a_+\psi_\e'(u)|\nabla u|^2\phi + \int a_+ \psi_\e(u)\nabla u\nabla \phi.
\end{equation*}
Owing to the fact that $\psi_\e'\geq 0$, we have that 
\begin{equation*}
\int a_+\psi_\e(u)\nabla u \nabla \phi \leq 0.
\end{equation*}
Following the argument in the proof of Proposition \ref{proposition:exist}, we derive that $\psi_\e(u)\nabla u \ra \nabla u^+$ weakly in $L^2(\Omega)$ as $\e\ra 0$, which in turn yields that
\begin{equation*}
\int a_+ \nabla u^+\nabla \phi\leq 0.
\end{equation*}
as desired. 

Along with the observation that $L_+u^+\geq 0$ in $\Omega$ in the sense of distribution, the inequality \eqref{eq:pos-msr} follows immediately from the interior energy estimate. We omit the details.
\end{proof}

Next lemma shows that there is no difference to define $\mu$ by $L_-u^-$. 

\begin{lemma}\label{lemma:same-msr} There holds 
\begin{equation*}
\int \phi d\mu = \int a_-\nabla u^-\nabla \phi,\quad\forall \phi\in C_c^\infty(\Omega).
\end{equation*}
\end{lemma}

\begin{proof} For any $\phi\in C_c^\infty(\Omega)$,
\begin{equation*}
\begin{split}
\int\phi d\mu^+-\int a_-\nabla u^-\nabla \phi &=-\int (a_+\nabla u^+-a_-\nabla u^-)\nabla\phi\\
&=-\int (a_+H(u)+a_-(1-H(u)))\nabla u\nabla \phi\\
&=-\int  A(x,u)\nabla u\nabla \phi\\
&=0.
\end{split}
\end{equation*}
\end{proof}

%
%

\section{Further Properties of the Free Boundary}\label{section:hausdorff}

In this section, we investigate some measure theoretic properties in regard of our free boundary, in preparation of the analysis in Section \ref{section:higher reg}. A key observation is that for $\mu$ almost every free boundary point, we can find a blowup limit which is a two plane solution, i.e., a function in form of $\alpha x_n^+ - \beta x_n^-$, after rotation. This idea was used and played a key role in \cite{AM}. Although the main arguments in this section  are similar to those in Section 4 of \cite{AM}, we will provide some  details, for the readers convenience. 

\begin{definition}\label{definition:scaled} Let $D$ be a domain in $\R^n$ and $v\in L^2(D)$. Define a function $v_{z,r}$ in $B_1$ by
\begin{equation*}
v_{z,r}(x):=
\begin{dcases}
\frac{v(rx+z)r^{\frac{n}{2}}}{\norm{v}_{L^2(B_r(z))}}&\text{if }\norm{v}_{L^2(B_r(z))}>0,\\
0&\text{otherwise},
\end{dcases}
\end{equation*}
provided that $B_r(z)\subset D$. 
\end{definition}

Let $u$ be a nontrivial weak solution to \eqref{eq:main} and $z\in\Gamma(u)$. By definition, $\norm{u_{z,r}}_{L^2(B_1)}\leq 1$ for any $0<r<\dist(z,\p\Omega)$. Therefore, there is a $v\in L^2(B_1)$ and a sequence $r_j\searrow 0$ such that $u_{z,r_j}\ra v$ weakly in $L^2(B_1)$ as $j\ra\infty$. In particular, $\norm{v}_{L^2(B_1)}\leq 1$. 

Note that $u_{z,r}$ is a weak solution to 
\begin{equation*}
\ddiv(A(rx+z,u_{z,r})\nabla u_{z,r})= 0\quad\text{in }B_1.
\end{equation*}
Utilizing the local energy estimate, we deduce that $u_{z,r_j}\ra v$ weakly in $W_{loc}^{1,2}(B_1)$ as $j\ra\infty$ up to a subsequence, and thus strongly in $L_{loc}^2(B_1)$. Likewise, by means of the interior H\"{o}lder estimate, we have $0<\gamma<1$ for which $u_{z,r_j}\ra v$ in $C_{loc}^{\gamma'}(B_1)$ as $j\ra\infty$ up to a subsequence, for any $0<\gamma'<\gamma$. As a result, $v\in W_{loc}^{1,2}(B_1)\cap C_{loc}^\gamma(B_1)\cap L^2(B_1)$.

As we follow the argument in the proof of Lemma \ref{lemma:Ca}, we observe that $v$ is a weak solution to 
\begin{equation*}
\ddiv(A_z(v)\nabla v) = 0\quad\text{in }B_1.
\end{equation*}
By defining a function $w$ in $B_1$ by 
\begin{equation*}
w(x) := a_+(z) v^+(x) - a_-(z) v^-(x),
\end{equation*}
$w$ becomes harmonic in $B_1$. It is noteworthy that the zero level surface of $v$ coincides with that of $w$, whence, inherits nice properties of nodal sets of harmonic functions. For instance, $v$ possesses the unique continuation property.\footnote{A function is said to have the unique continuation property, if it does not vanish on any open subset, unless it vanishes in the entire domain.} 

\begin{definition}\label{definition:blowup} Given $z\in\Gamma(u)$, define $Blo(u,z)$ by the class of all possible strong $L_{loc}^2(B_1)$ limits of $\{u_{z,r}\}_{r>0}$.
\end{definition}

The following lemma tells us that any free boundary point of a blowup limit admits a convergent sequence of the original free boundary points, unless the blowup limit is trivial.

\begin{lemma}\label{lemma:dist} Let $S\subset\Gamma(u)$ with $\mu(S)>0$. Then the following is true for $\mu$-a.e. $z\in S$: For any $v\in Blo(u,z)\setminus\{0\}$ and $x\in\p\{v>0\}\cap B_1$, 
\begin{equation}\label{eq:dist}
\varliminf_{r\ra 0}\frac{\dist(z+rx,S)}{r} = 0.
\end{equation}
\end{lemma}

\begin{proof} Since $\mu$ is a Radon measure (Lemma \ref{lemma:pos-msr}) and $\mu(S)>0$, $\mu$-a.e point is a density point of $S$; i.e., for $\mu$-a.e. $z\in S$, we have 
\begin{equation}\label{eq:density}
\lim_{r\ra 0}\frac{\mu(S\cap B_r(z))}{\mu(B_r(z))} = 1.
\end{equation}
We prove that all density points of $S$ satisfies \eqref{eq:dist}. 

Suppose towards a contradiction that there is some $v\in Blo(u,z)$ and $x\in\p\{v>0\}\cap B_1$ such that \eqref{eq:dist} fails. Then we may take a small number $\e>0$ and $r_\e>0$ such that $\e+|x|<1$ and 
\begin{equation}\label{eq:dist-false}
\frac{\dist(z+ rx,S)}{r} \geq \e\quad\text{for any }0<r\leq r_\e.
\end{equation}
In particular, we have that $B_{\e r}(z+rx)\cap S=\emptyset$. 

Select $0<\rho<1$ such that $\e+|x|<\rho$. Then $B_{\e r}(z+rx)\subset B_{\rho r}(z)$. On the other hand, from \eqref{eq:density} one may deduce that 
\begin{equation*}
\lim_{r\ra 0}\frac{\mu(B_{\rho r}\setminus S)}{\mu(B_{\rho r}(z))} = 0.
\end{equation*}
Combining these two observations, we arrive at
\begin{equation}\label{eq:msr-density-false}
\lim_{r\ra 0}\frac{\mu(B_{\e r}(z+rx))}{\mu(B_{\rho r}(z))} = 0.
\end{equation}

Now let us take a sequence $r_j\searrow 0$ along which $u_{z,r_j}\ra v$ strongly in $L_{loc}^2(B_1)$. As observed earlier, we may without loss of generality say that $u_{z,r_j}\ra v$ weakly in $W_{loc}^{1,2}(B_1)$. As we denote by $\mu_j$ the measure $d\mu_j = \ddiv(a_+(r_jx+z)\nabla u_{z,r_j})$ and by $\mu_0$ the measure $d\mu_0 = a_+(z)\Delta v^+$, one may easily deduce from the weak $W_{loc}^{1,2}(B_1)$ convergence that $\mu_j \ra \mu_0$ locally in $B_1$ in the sense of Radon measure.\footnote{That is,  
\begin{equation*}
\int \phi d\mu_j \ra \int \phi d\mu_0\quad\text{for any }\phi\in C_c^\infty(B_1).
\end{equation*}
}

As both $B_\e(x)$ and $B_\rho$ being compactly contained in $B_1$, it follows from the above convergence that $\mu_v(B_\e(x)) = \lim_{j\ra\infty}\mu_{u_{z,r_j}}(B_\e(x))<\infty$ and $\mu_v(B_\rho)= \lim_{j\ra\infty}\mu_{u_{z,r_j}}(B_\rho)<\infty$. Moreover, we have $\mu_v(B_\e(x))>0$. Otherwise, we have $\mu_v(B_\e(x))=0$, which implies that $v^+$ is a nonnegative harmonic function in $B_\e(x)$, and thus, by the maximum principle, $v^+\equiv 0$ in $B_\e(x)$. Due to Lemma \ref{lemma:same-msr}, however, one may also deduce from $\mu_v(B_\e(x))=0$ that $v^-$ is a nonnegative harmonic function in $B_\e(x)$, whence showing that $v^-\equiv 0$ in $B_\e(x)$. It follows that $v\equiv 0$ in $B_\e(x)$, and thus by the unique continuation property of $v$, we arrive at $v\equiv 0$ in $B_1$, a contradiction to our initial choice of $v$. Similarly, we may prove $\mu_v(B_\rho)>0$. However, \eqref{eq:msr-density-false} implies that 
\begin{equation*}
\frac{\mu_v(B_\e(x))}{\mu_v(B_\rho)} = \lim_{j\ra\infty}\frac{\mu_{u_{z,r_j}}(B_\e(x))}{\mu_{u_{z,r_j}}(B_{\rho})} =\lim_{j\ra\infty}\frac{\mu(B_{\e r_j}(z+r_jx))}{\mu(B_{\rho r_j}(z))} = 0,
\end{equation*} 
a contradiction. Thus, the lemma is proved. 
\end{proof} 

With the previous lemma at hand, we observe that blowup at a free boundary point of a blowup limit indeed belongs to a limit of the original blowup sequence. This is certainly true for harmonic functions. Moreover, it becomes intuitively clear that only nondegenerate points enjoy this property, unless degenerate points carry some positive measure. 

\begin{lemma}\label{lemma:blowup-twice} Let $0<\rho<1$. The following is true for $\mu$-a.e. $z\in\Gamma(u)$: For any $v\in Blo(u,z)$ and $x\in\p\{v>0\}\cap B_{1-\rho}$, there holds $v_{x,\tau\rho}\in Blo(u,z)$ for any $0<\tau\leq 1$.
\end{lemma}

\begin{proof} Fix $0<\rho<1$. Note that the statement is clearly true if $x=0$, since if $u_{z,r_i}\ra v$ in $L_{loc}^2(B_1)$ as $i\ra\infty$, 
\begin{equation*}
u_{z,\rho r_i} = (u_{z,r_i})_{0,\rho} \ra v_{0,\rho}\quad\text{in }L^2(B_1)\text{ as $i\ra\infty$}.
\end{equation*}
Therefore, we only need to consider $x\in\p\{v>0\}\cap(B_\rho\setminus\{0\})$. 

In addition, one may argue in a similar manner as above and deduce that if $v_{x,\rho}\in Blo(u,z)$, then $v_{x,\tau\rho}\in Blo(u,z)$ for any $0<\tau\leq 1$. Thus, it suffices to prove the statement of Lemma \ref{lemma:blowup-twice} for $\tau=1$. 

For any triple $(j,k,l)$ of positive integer, define $S_{j,k,l}$ by the subset of $\Gamma(u)$ such that for each $z\in S_{j,k,l}$, $\dist(z,\p\Omega)\geq\frac{1}{l}$ and there correspond $v^z\in Blo(u,z)$ and $x^z\in\p\{v>0\}\cap B_\rho$ such that 
\begin{equation}\label{eq:Skl}
\norm{u_{z,r} - v_{x^z,\rho}^z}_{L^2(B_{1/(2j)})}>\frac{1}{k},\quad\forall r\leq\frac{1}{2l}.
\end{equation}
Note that for any $z\in S_{j,k,l}$, the corresponding $v^z$ is not identically zero in $B_1$ because any scaled version of a trivial function was defined to be trivial. 

Assume to the contrary that $\mu(S_{j,k,l})>0$ for some positive integers $j_0$, $k_0$ and $l_0$, and let us write $S_0 = S_{j_0,k_0,l_0}$.

Set $\cA:=\{v_{x^z,\rho}^z:z\in S_0\}$. Note that $\cA$ is uniformly bounded in $W^{1,2}(B_{1/(2j_0)})$. As $W^{1,2}(B_{1/(2j_0)})$ being compactly embedded in $L^2(B_{1/(2j_0)})$, $\cA$ can be covered by a countable collection $\{G_m\}_{m=1}^\infty$ of balls in $L^2(B_{1/(2j_0)})$ with $\diam(G_m)\leq\frac{1}{2k_0}$; i.e., $\norm{w-\phi}_{L^2(B_{1/(2j_0)})}\leq\frac{1}{2k_0}$ whenever $w,\phi\in G_m$. 

As $\cA$ being covered by $\{G_m\}_{m=1}^\infty$, we may decompose $S_0$ into a countable union of $S_0^m$, $m=1,2,\cdots$, where $S_0^m:=\{z\in S_0:v_{x^z,\rho}^z\in G_m\}$. From the assumption that $\mu(S_0)>0$, there exists some $m_0$ such that $\mu(S_0^{m_0})>0$. 

For the notational convenience, let us abbreviate by $S$ the set $S_0^{m_0}$. Owing to Lemma \ref{lemma:dist}, we may choose $z\in S$ such that \eqref{eq:dist} is true. Since $|x^z|<\rho$, we are able to take $0<\sigma<1$ such that $|x^z|+(1-\rho)<\sigma$. Now we denote by $\{r_i\}_{i=1}^\infty$ the sequence along which $u_{z,r_i}\ra v_{x^z,\rho}^z$ strongly in $L_{loc}^2(B_1)$, and thus, in $L^2(B_\sigma)$, as $i\ra\infty$. By means of \eqref{eq:dist}, one may assign a point $z^i\in S$ for each $i=1,2,\cdots$ such that 
\begin{equation*}
0=\lim_{i\ra\infty}\frac{|z^i - (z + r_ix^z)|}{r_i} = \lim_{i\ra\infty}\left|\frac{z^i - z}{r_i} - x^z\right|.
\end{equation*}

Take a sufficiently large $i_0$ such that $|\frac{z^i-z}{r_i}|\leq\sigma-(1-\rho)$ for all $i\geq i_0$. Then $\frac{z^i-z}{r_i}+\rho y\in B_\sigma$ for any $y\in B_1$, which implies that 
\begin{equation*}
u_{z^i,\rho r_i} = (u_{z,r_i})_{\frac{z^i-z}{r_i},\rho}\ra v_{x^z,\rho}^z\quad\text{strongly in }L^2(B_1).
\end{equation*}
Hence, by making $i_0$ even larger if necessary, we have  
\begin{equation}\label{eq:uzj-conv}
\norm{u_{z^i,\rho r_i} - v_{x^z,\rho}^z}_{L^2(B_1)}\leq \frac{1}{2k_0},\quad\forall i\geq i_0,
\end{equation}
where $k_0$ is the integer chosen in the beginning of this proof. 

However, we observe that
\begin{equation}\label{eq:gm0}
\norm{v_{x^{z^i},\rho}^{z^i}- v_{x^z,\rho}^z}_{L^2(B_{1/(2j_0)})}\leq\frac{1}{2k_0},\quad\forall i=1,2,\cdots,
\end{equation}
since all $z^i$'s and $z$ belong to $S=S_0^{m_0}$; recall that $\diam(G_{m_0})\leq\frac{1}{2k_0}$. Combining \eqref{eq:gm0} with \eqref{eq:uzj-conv}, we arrive at
\begin{equation*}
\norm{u_{z^i,\rho r_i}-v_{x^{z^i},\rho}^{z^i}}_{L^2(B_{1/(2j_0)})}\leq\frac{1}{k_0},\quad\forall i\geq i_0.
\end{equation*}
This contradicts our assumption in \eqref{eq:Skl} as we further enlarge $i_0$ so that $r_i\leq\frac{1}{2l_0}$ for all $i\geq i_0$. This completes the proof.
\end{proof} 

Next we prove that degenerate points with vanishing order greater than or equal to 2 have $\mu$-measure zero. 

\begin{lemma}\label{lemma:D} Let $\cD(u)$ be a subset of $\Gamma(u)$ such that for each $z\in\cD(u)$,  
\begin{equation}\label{eq:D}
\varlimsup_{r\ra 0}\frac{\norm{u}_{L^2(B_r(z))}}{r^{\frac{n}{2}+2}}<\infty.
\end{equation}
Then $\mu(\cD(u))=0$. 
\end{lemma}

\begin{proof} Owing to \eqref{eq:pos-msr}, we have, for any $z\in\cD(u)$, 
\begin{equation}\label{eq:Dk-re}
\varlimsup_{r\ra 0}\frac{\mu(B_r(z))}{r^n}<\infty.
\end{equation}
Hence, the measure $\mu$ restricted on $\cD(u)$ is absolutely continuous with respect to $\cL^n$, the $n$-dimensional Lebesgue measure.

Let us consider a subset $\cD'(u)$ of $\cD(u)$ consisting of all $z$ at which 
\begin{equation*}
\varlimsup_{r\ra 0}\frac{\norm{u}_{L^2(B_r(z))}}{r^{\frac{n}{2}+3}}<\infty.
\end{equation*} 
Following the argument right above, one may also observe that the measure $\mu$ restricted on $\cD'(u)$ is absolutely continuous with respect to $\cH^{n+1}$, the $(n+1)$-dimensional Hausdorff measure. Evidently, $\cH^{n+1}(\cD'(u))=0$, from which it follows that $\mu(\cD'(u))=0$. 

Thus, it will suffice to prove $\cL^n(\cD(u)\setminus\cD'(u))=0$ to ensure that $\mu(\cD(u))=0$. Suppose towards a contradiction that $\cL^n(\cD(u)\setminus\cD'(u))>0$. Then there exists a density point $z\in\cD(u)\cap\cD'(u)$ with respect to $\cL^n$. As $\cD(u)\cap\cD'(u)\subset\Gamma(u)$, it follows that 
\begin{equation}\label{eq:dense-D}
\lim_{r\ra 0}\frac{\cL^n(\Gamma(u)\cap B_r(z))}{\cL^n(B_r(z))} = 1.
\end{equation}
However, as $z$ being chosen from $\cD(u)\setminus\cD'(u)$, we have 
\begin{equation*}
\varlimsup_{r\ra 0}\frac{\norm{u}_{L^2(B_r(z))}}{r^{\frac{n}{2}+3}}=\infty,
\end{equation*}
by which we may take a sequence $r_j\searrow 0$ such that  
\begin{equation}\label{eq:doubling-D}
\norm{u}_{L^2(B_{r_j}(z))}\leq 2^3\norm{u}_{L^2(B_{r_j/2}(z))}.
\end{equation}

Taking a further subsequence of $\{r_j\}_{j=1}^\infty$ if necessary, we obtain a strong $L_{loc}^2(B_1)$ limit $v$ of $\{u_{z,r_j}\}_{j=1}^\infty$; see the discussion before Lemma \ref{lemma:dist}. With \eqref{eq:doubling-D} at hand, $v$ satisfies 
\begin{equation*}\label{eq:D-v}
\norm{v}_{L^2(B_{1/2})}\geq 2^{-3}.
\end{equation*}
Without loss of generality, we can take $x^0\in\bar{B}_{1/2}$ such that $v(x^0)\geq c_n$. Then the local H\"{o}lder regularity of $v$ allows us to choose a small constant $0<\delta<\frac{1}{4}$ such that $\sup_{B_\delta(x^0)}v\geq \frac{c_n}{2}$. Utilizing the locally uniform convergence, we may take a sufficiently large $j_\delta$ such that $\sup_{B_\delta(x^0)} u_{z,r_j}\geq \frac{c_n}{3}$ for all $j\geq j_\delta$, which suffices to prove that $B_{\delta}(x^0)\subset \Omega^+(u_{z,r_j})\cap B_1$ for all $j\geq j_\delta$. Alternatively, we have $B_{\delta r_j}(r_jx^0+z)\subset\Omega^+(u)\cap B_{r_j}(z)$ for any $j\geq j_\delta$, whence
\begin{equation*}
\varliminf_{j\ra\infty}\frac{\cL^n(B_{r_j}(z)\setminus\Gamma(u))}{\cL^n(B_{r_j}(z))}\geq\varliminf_{j\ra\infty}\frac{\cL^n(B_{\delta r_j}(r_jx^0+z))}{\cL^n(B_{r_j}(z))} = \delta^n>0,
\end{equation*}
which violates \eqref{eq:dense-D}. Hence, the proof is finished.
\end{proof}

We finish this section by proving that for $\mu$ almost every point of the free boundary has a two plane solution as its blowup limit. Due to the preceding lemma, we only need to consider those points which have vanishing order less than 2. This will give us a doubling condition, as observed in the proof of Lemma \ref{lemma:D}, so that we may find a nontrivial blowup limit.

\begin{lemma}\label{lemma:flat-ae} For $\mu$-a.e. $z\in\Gamma(u)$, there exists a function $v\in Blo(u,z)$ such that 
\begin{equation*}
v(x) = \beta_+(x\cdot\nu)^+ - \beta_-(x\cdot\nu)^-\quad\text{in }B_1,
\end{equation*}
where $\beta_\pm = \beta/a_\pm(z)$ for some $\beta>0$ and $\nu$ is a unit vector in $\R^n$.
\end{lemma}

\begin{proof} With Lemma \ref{lemma:D} at hand, it suffices to prove this lemma for $z\in\Gamma(u)\setminus\cD(u)$.

Since $z\notin\cD(u)$, one may follow the argument in the proof of Lemma \ref{lemma:D} and obtain a nontrivial $v\in Blo(u,z)$. As discussed in the beginning of this section, the function $w$, defined by $w(x):=a_+(z)v^+(x) - a_-(z)v^-(x)$, is harmonic in $B_1$. Hence, we may find $x^0\in\p\{w>0\}\cap B_1$ where $|\nabla w(x^0)|>0$. Taking $\nu:=\frac{\nabla w(x^0)}{|\nabla w(x^0)|}$ and $\beta = |\nabla w(x^0)|$,  one may easily deduce that as $\rho\searrow 0$, 
\begin{equation*}
w_{x^0,\rho}\ra \beta(x\cdot\nu)\quad\text{strongly in $L_{loc}^2(B_1)$},
\end{equation*} 
or alternatively,
\begin{equation*}
v_{x^0,\rho}\ra \beta_+(x\cdot\nu)^+ - \beta_-(x\cdot\nu)^-\quad\text{strongly in $L_{loc}^2(B_1)$}.
\end{equation*}
The proof  now follows by  Lemma \ref{lemma:blowup-twice}.
\end{proof}

%
%

\section{$C^{1,\alpha}$ Regularity of Nondegenerate Free Boundaries}\label{section:higher reg}

This section is devoted to the investigation of higher regularity of our free boundaries. According to Lemma \ref{lemma:flat-ae}, it is natural to start with those points which have two plane solutions as blowup limits. We may prove later that our free boundary satisfies a flatness condition (e.g., \eqref{eq:assump-flat}) around such points, and it turns out that such a flatness condition can be improved in an inductive way (Lemma \ref{lemma:approximation}). After all, we observe that those points admit a small neighborhood in which the free boundary is a $C^{1,\alpha}$ graph.

\begin{theorem}\label{theorem:C1a-fb} Let $a_+$ and $a_-$ satisfy \eqref{eq:cont} with an $\alpha$-H\"{o}lder continuous $\omega$, and $u$ be a weak solution to \eqref{eq:main} in $\Omega$. Then for $\mu$-a.e. $z\in\Gamma(u)$, there is $r>0$ such that $\Gamma(u)\cap B_r(z)$ is a $C^{1,\alpha}$ graph; here the radius $r$ and the $C^{1,\alpha}$ character of the graph may depend on $u$, $z$ and the H\"{o}lder norm of $a_\pm$.
\end{theorem}

For further analysis on the measure theoretic regularity of our free boundary, let us define the classes of nondegenerate and degenerate points.

\begin{definition}\label{definition:N-S} Let $u$ be a weak solution to \eqref{eq:main} in $\Omega$. Define $\cN(u)$ and $\cS(u)$ by the classes of nondegenerate points and respectively degenerate points, i.e.,
\begin{equation*}
\cN(u) := \left\{z\in\Gamma(u): \varlimsup_{r\ra 0}\frac{\norm{u}_{L^2(B_r)}}{r^{\frac{n}{2}+1}}>0\right\},
\end{equation*}
and
\begin{equation*}
\cS(u) := \Gamma(u)\setminus \cN(u) = \left\{z\in\Gamma(u): \varlimsup_{r\ra 0}\frac{\norm{u}_{L^2(B_r)}}{r^{\frac{n}{2}+1}}=0\right\}.
\end{equation*}
\end{definition}

Theorem \ref{theorem:C1a-fb} implies that our free boundary points are essentially nondegenerate, which opens up a way to characterize the measure $\mu$ with respect to $(n-1)$-dimensional Hausdorff measure.

\begin{theorem}\label{theorem:Hn-1} Under the assumption of Theorem \ref{theorem:C1a-fb}, $\spt(\mu)$ has $\sigma$-finite $(n-1)$-dimensional Hausdorff measure. More specifically, $\cN(u)$ has $\sigma$-finite $(n-1)$-dimensional Hausdorff measure, while $\cS(u)$ has $\mu$-measure zero. In addition, the conclusion of Theorem \ref{theorem:C1a-fb} holds for every $z\in\cN(u)$, and $\cN(u)$ is relatively open in $\Gamma(u)$. 
\end{theorem}

Let us observe an elementary fact for the future reference. 

\begin{lemma}\label{lemma:comparison} Let $D$ be a domain in $\R^n$, $1\leq p\leq\infty$ and $f,g\in L^p(D)$. Then for any given positive numbers $\beta_+$ and $\beta_-$, there holds
\begin{equation*}
\norm{(\beta_+f^+ -\beta_-f^-) - (\beta_+g^+ - \beta_-g^-)}_{L^p(D)}\leq (\beta_+ + \beta_-)\norm{f-g}_{L^p(D)}.
\end{equation*}
\end{lemma}

\begin{proof} Set $\delta:=\norm{f-g}_{L^p(D)}$. We only need to consider the inequality on $D\cap\{f\geq 0\}\cap\{g\leq 0\}$ and $D\cap \{f\leq0\}\cap \{g\geq 0\}$. Due to the symmetry of the proof, let us only focus on the former set, which we denote by $E$, for notational convenience. Then $0\leq -g \leq f-g$ and $0\leq f\leq f-g$ on $E$, which implies that $\norm{g}_{L^p(E)}\leq\delta$ and $\norm{f}_{L^p(E)}\leq \delta$. Therefore,
\begin{equation*}
\begin{split}
\norm{(\beta_+f^+ -\beta_-f^-) - (\beta_+g^+ - \beta_-g^-)}_{L^p(E)} &= \norm{\beta_+ f - \beta_-g}_{L^p(E)}\\
&\leq (\beta_+ + \beta_-)\delta,
\end{split}
\end{equation*}
and thus the proof is finished.
\end{proof}

In what follows we are going to use Lemma \ref{lemma:comparison} with $p=2$ or $\infty$. Let us introduce a notation for two plane solutions.

\begin{definition}\label{definition:two plane} Given $\beta>0$, $z\in\R^n$ and a vector $\nu$ in $\R^n$, define a function $P_{\beta,\nu}^z$ on $\R^n$ by
\begin{equation*}
P_{\beta,\nu}^z(x):=\frac{\beta}{a_+(z)}(x\cdot\nu)^+-\frac{\beta}{a_-(z)}(x\cdot \nu)^-,
\end{equation*}
In particular, we denote $P_{\beta,\nu}^z$ by $P_{\beta,\nu}$. 
\end{definition}

\begin{lemma}\label{lemma:approximation}
Let $u$ be a weak solution to \eqref{eq:main} in $B_1$ satisfying $0\in \Gamma(u)$, $\norm{\nabla u}_{L^\infty(B_1)}\leq 1$ and 
\begin{equation}\label{eq:assump-flat}
\norm{u - P_{\beta,\nu}}_{L^2(B_1)}\leq\e,
\end{equation}
for some $\beta,\e>0$ and a nonzero vector $\nu\in\R^n$. There are (small) positive universal constants $\eta$ and $\bar{r}$ such that if $a_+$ and $a_-$ satisfy
\begin{equation}\label{eq:assump-coeff}
|a_\pm(x) - a_\pm(0)|\leq \e\eta|x|^\alpha\quad\text{in }B_1, 
\end{equation}
then there exists a sequence $\{\nu^k\}_{k=0}^\infty$ of vectors in $\R^n$ such that 
\begin{equation}\label{eq:improve-flat}
\frac{1}{\bar{r}^{n/2}}\norm{u - P_{\beta,\nu^k}}_{L^2(B_{\bar{r}^k})}\leq \e\bar{r}^{k(1+\alpha)},\quad\forall k=0,1,2,\cdots,
\end{equation}
and that with a universal constant $c_0>0$,
\begin{equation}\label{eq:nu}
|\nu^k-\nu^{k-1}|\leq \frac{c_0}{\beta}\e\bar{r}^{k\alpha},\quad\forall k=1,2,\cdots.
\end{equation}
\end{lemma}

\begin{proof} From \eqref{eq:assump-flat} observe that the initial case for \eqref{eq:improve-flat} is satisfied by simply setting $\nu^0=\nu$. 

Now we suppose that \eqref{eq:improve-flat} and \eqref{eq:nu} are met for some $k\geq 0$. Then define a function $u_k$ in $B_1$ by $u_k(x):= u(\bar{r}^kx)/\bar{r}^k$. Since $\norm{\nabla u}_{L^\infty(B_1)}\leq 1$, we obtain $\norm{\nabla u_k}_{L^\infty(B_1)}=\norm{\nabla u}_{L^\infty(B_{\bar{r}^k})}\leq 1$. 

Consider functions $v_k$ and $w_k$ defined in $B_1$ by 
\begin{equation*}
v_k(x):=a_+(0)u_k^+(x)-a_-(0)u_k^-(x),
\end{equation*}
and respectively by
\begin{equation*}
w_k(x):=\frac{v_k(x)-\beta(x\cdot\nu^k)}{\e\bar{r}^{k\alpha}}.
\end{equation*}
By the induction hypothesis on \eqref{eq:improve-flat} and Lemma \ref{lemma:comparison} we have 
\begin{equation}\label{eq:w-bound}
\norm{w_k}_{L^2(B_1)}\leq\frac{2}{\lambda}.
\end{equation}
On the other hand, one may observe from \eqref{eq:main} that $w_k$ solves
\begin{equation}\label{eq:w-equation}
\Delta w_k=\frac{1}{\e\bar{r}^{k\alpha}}\ddiv(\sigma_k(x)\nabla u_k)\quad\text{in $B_1$}
\end{equation}
in the sense of distribution, where
\begin{equation*}
\sigma_k(x):=(a_+(0)-a_+(\bar{r}^kx))H(u_k)+(a_-(0)-a_-(\bar{r}^kx))(1-H(u_k)).
\end{equation*}
Since $\norm{\nabla u_k}_{L^\infty(B_1)}\leq 1$, we deduce from \eqref{eq:assump-coeff} that 
\begin{equation}\label{eq:small-rhs}
\frac{1}{\e\bar{r}^{k\alpha}}\norm{\sigma_k\nabla u_k}_{L^\infty(B_1)}\leq \eta.
\end{equation}

Now consider the harmonic replacement $h$ of $w_k$; i.e., $\Delta h=0$ in $B_1$ with $h=w$ on $\p B_1$. With \eqref{eq:small-rhs} at hand, we may apply the global $L^\infty$, and thus, $L^2$ estimate (e.g., Theorem 8.15 in \cite{GT}) to \eqref{eq:w-equation} and obtain 
\begin{equation}\label{eq:w-h}
\norm{w_k-h}_{L^2(B_1)}\leq c_1\eta.
\end{equation}

Due to \eqref{eq:w-bound} and \eqref{eq:w-h}, it follows that $\norm{h}_{L^2(B_1)}\leq c_2$; here we assume that $\eta<1$, which will be fulfilled later. By the interior estimates for derivatives of harmonic functions, we know that $|\nabla h(0)|\leq c_3$ and $\norm{D^2h}_{L^\infty(B_{1/2})}\leq c_4$. Let us denote by $l(x)$ the linear function $\nabla h(0)\cdot x$. Then, the Taylor expansion yields that if $0<r\leq\frac{1}{2}$,
\begin{equation*}
\frac{1}{\bar{r}^{n/2}}\norm{h - h(0) - l}_{L^2(B_{\bar{r}})}\leq\frac{1}{2}c_4r^2.
\end{equation*}
On the other hand, since $w_k(0) = 0$, it follows from \eqref{eq:w-h} that $|h(0)| = |h(0) - w_k(0)|\leq c_1\eta$. Therefore, as we denote the linear function $\nabla h(0)\cdot x$ by $l(x)$, we derive that
\begin{equation*}
\begin{split}
\frac{1}{\bar{r}^{n/2}}\norm{w_k - l}_{L^2(B_{\bar{r}})}&\leq \frac{\lambda}{2}\bar{r}^{1+\alpha},
\end{split}
\end{equation*}
as we choose $\bar{r}$ small enough such that $2c_4\bar{r}^2\leq \lambda\bar{r}^{1+\alpha}$, and accordingly select $\eta$ so as to satisfy $8c_1\eta\leq\lambda\bar{r}^{1+\alpha}$. 

Let us define $\nu^{k+1}:= \nu^k + \e\bar{r}^{k\alpha}\beta^{-1}\nabla h_0(0)$. Then as we rephrase the above inequality in terms of $v_k$ and then apply Lemma \ref{lemma:comparison}, we arrive at 
\begin{equation*}
\frac{1}{\bar{r}^{(k+1)n/2}}\norm{u - P_{\beta,\nu^{k+1}}}_{L^2(B_{\bar{r}^{k+1}})}\leq \e\bar{r}^{(k+1)(1+\alpha)}.
\end{equation*}
This inequality is exactly \eqref{eq:improve-flat} with $k$ replaced by $k+1$. Owing to the fact that $|\nabla h(0)|\leq c_3$, the inequality in \eqref{eq:nu} is also true again for $k+1$ instead of $k$. The proof is finished by the induction principle.
\end{proof}

An immediate consequence of the preceding lemma is that if $u$ is close to a two plane solution around a free boundary point, then the free boundary can be locally trapped in between two $C^{1,\alpha}$ graph. We skip the proof; one can easily deduce it by making use of Lemma \ref{lemma:comparison}.

\begin{corollary}\label{corollary:approximation} Under the circumstance of Lemma \ref{lemma:approximation}, there exists a unit vector $e$ in $\R^n$ such that 
\begin{equation*}
\frac{1}{r^{n/2}}\norm{u - P_{\beta,e}}_{L^2(B_r)}\leq C\e r^{1+\alpha},
\end{equation*}
and in particular, 
\begin{equation*}
\Gamma(u)\cap B_1\subset \{x\in B_1: |x\cdot e|\leq C\e\beta^{-1}|x|^{1+\alpha}\}.
\end{equation*}
\end{corollary}

\begin{lemma}\label{lemma:flat-nbd} Let $u$ be a weak solution to \eqref{eq:main} in $B_1$ with $\norm{u}_{L^2(B_1)}\leq 1$. Suppose that $u$ satisfies 
\begin{equation}\label{eq:assump-flat-nbd}
\norm{u - P_{\beta,\nu}}_{L^2(B_{7/8})}\leq\delta
\end{equation}
and 
\begin{equation}\label{eq:assump-coeff-flat-nbd}
|a_\pm(x) - a_\pm(0)|\leq \delta|x|^\alpha\quad\text{in }B_1.
\end{equation}
for some positive numbers $\beta$ and $\delta<1$, and a unit vector $\nu\in\R^n$. Then for $0<\rho\leq\frac{3}{8}$ satisfying $\rho^{1+\alpha}\leq\frac{c_1}{\beta}$, there holds for any $z\in B_\rho\cap\Gamma(u)$, $B_\rho(z)\subset B_{3/4}$ and 
\begin{equation}\label{eq:flat-nbd}
\norm{\frac{u(\rho \cdot + z)}{\rho} - P_{\beta,\nu}^z}_{L^2(B_1)}\leq \frac{c_2\delta}{\rho}. 
\end{equation} 
\end{lemma}

\begin{proof} From \eqref{eq:assump-flat-nbd} and Lemma \ref{lemma:comparison}, we obtain 
\begin{equation}\label{eq:assump-flat-nbd-v}
\norm{v - l}_{L^2(B_{7/8})}\leq \frac{2\delta}{\lambda},
\end{equation}
where $l(x) := \beta(x\cdot\nu)$. As observed in the proof of Lemma \ref{lemma:approximation}, 
\begin{equation*}
\Delta (v - l) = \ddiv(\sigma(x)\nabla u)\quad\text{in $B_{7/8}$},
\end{equation*}
where $\sigma(x) := (a_+(0) - a_+(x))H(u) + (a_-(0) - a_-(x))(1-H(u))$ in $B_1$. Due to the assumption that $\norm{u}_{L^2(B_1)}\leq 1$, we observe from Theorem \ref{theorem:opt reg} that $\norm{\nabla u}_{L^\infty(B_{7/8})}\leq c_1$. In combination with \eqref{eq:assump-coeff-flat-nbd}, we derive that 
\begin{equation}\label{eq:assump-coeff-flat-nbd-sigma}
\norm{\sigma\nabla u}_{L^\infty(B_{7/8})}\leq c_1\delta. 
\end{equation}
Utilizing \eqref{eq:assump-flat-nbd-v} and \eqref{eq:assump-coeff-flat-nbd-sigma}, we deduce from the local boundness of weak solutions (Theorem 8.18 in \cite{GT}) that 
\begin{equation*}
\norm{v-l}_{L^\infty(B_{3/4})}\leq c_2\left(\norm{v-l}_{L^2(B_{7/8})} + \norm{\sigma\nabla u}_{L^\infty(B_{7/8})}\right)\leq c_3\delta.
\end{equation*} 
Making use of Lemma \ref{lemma:comparison} once again, we arrive at 
\begin{equation}\label{eq:assump-flat-nbd-linf}
\norm{u - P_{\beta,\nu}}_{L^\infty(B_{3/4})}\leq c_0\delta.
\end{equation}

From \eqref{eq:assump-flat-nbd-linf}, we deduce that for any $z\in \Gamma(u)\cap B_{3/4}$, 
\begin{equation}\label{eq:flat-nbd-re}
\beta|z\cdot\nu|\max\{a_+(0)^{-1},a_-(0)^{-1}\}\leq c_0\delta.
\end{equation}
Therefore, one may deduce as in Lemma \ref{lemma:comparison} that if $x\in B_1$, then 
\begin{equation}\label{eq:x+z-x}
\left|P_{\beta,\nu}\left(x+\frac{z}{\rho}\right) - P_{\beta,\nu}(x)\right|\leq\frac{3c_0\delta}{\rho}.
\end{equation}

On the other hand, as we restrict $z\in\Gamma(u)\cap B_\rho$, we derive from \eqref{eq:assump-coeff-flat-nbd} that for any $x\in B_1$, 
\begin{equation}\label{eq:coeff-close}
\begin{split}
|P_{\beta,\nu}^z(x) - P_{\beta,\nu}(x)|\leq\frac{\beta\delta\rho^\alpha}{\lambda^2}\leq\frac{c_0\delta}{\rho},
\end{split}
\end{equation}
by choosing $\rho$ sufficiently small such that $\beta\rho^{1+\alpha}/\lambda^2\leq c_0$; here we have used the assumption that $\nu$ is a unit vector to estimate $|x\cdot\nu|\leq 1$. Combining \eqref{eq:x+z-x} and \eqref{eq:coeff-close} together, we arrive at \eqref{eq:flat-nbd} by means of the triangle inequality; we omit the details. 
\end{proof}

We are now in position to prove Theorem \ref{theorem:C1a-fb}. 

\begin{proof}[Proof of Theorem \ref{theorem:C1a-fb}] For definiteness, let $\omega$ in \eqref{eq:cont} to be   $\omega(r)=\omega_0 r^\alpha$ with $\omega_0>0$ and $0<\alpha<1$. 

From Lemma \ref{lemma:flat-ae}, for $\mu$-a.e. $z\in\Gamma(u)$, there is a sequence $r_j\ra 0$ such that $u_{z,r_j}\ra P_{\beta,\nu}$ in $L_{loc}^2(B_1)$ for some $\beta>0$ and a unit vector $\nu\in\R^n$. 

Now fix $\delta>0$ and we take a sufficiently large $j_0$ such that 
\begin{equation*}
\norm{u_{z,r_j} - P_{\beta,\nu}}_{L^2(B_{7/8})}\leq \delta,\quad\forall j\geq j_\delta. 
\end{equation*}
By definition, for any $j=1,2,\cdots$, $\norm{u_{z,r_j}}_{L^2(B_1)} = 1$ and $u_{z,r_j}$ solves
\begin{equation*}
\ddiv(A(r_jx+z,u_{z,r_j})\nabla u_{z,r_j}) = 0\quad\text{in }B_1.
\end{equation*}
Since $\omega$ is H\"{o}lder, and thus, Dini continuous, we may invoke Theorem \ref{theorem:opt reg} and assert that 
\begin{equation}\label{eq:uzrj-L}
\norm{\nabla u_{z,r_j}}_{L^\infty(B_{7/8})}\leq L,
\end{equation}
where $L>0$ depends only on $n$, $\lambda$, $\omega_0$ and $\alpha$. It is also clear that 
\begin{equation}\label{eq:apm-rj}
|a_\pm(r_jx + z) - a_\pm(z)|\leq \omega_0r_j^\alpha|x|^\alpha\quad\text{in }B_1.
\end{equation}

Select an integer $k_\delta$ larger than $j_\delta$ such that 
\begin{equation}\label{eq:omega 0 rj1}
\omega_0 r_{k_\delta}^\alpha\leq \delta.
\end{equation}
Then applying Lemma \ref{lemma:flat-nbd} to $u_{z,r_j}$ and $a_{\pm}(r_j\cdot + z)$, we obtain $0<\rho<\frac{3}{8}$ with $\rho\leq \frac{c_1}{\beta}$, such that for any $\xi\in B_\rho(z)\cap\Gamma(u)$, $B_\rho(\xi)\subset B_{3/4}$ and 
\begin{equation}\label{eq:xi-rho}
\norm{\frac{u_{z,r_{k_\delta}}(\rho \cdot+\xi)}{\rho} - P_{\beta,\nu}^\xi}_{L^2(B_1)}\leq\frac{c_2\delta}{\rho}.
\end{equation}

Now let us fix $\xi\in B_\rho(z)\cap \Gamma(u)$, and for notational convenience, define a function $\tilde{u}$ on $B_1$ by 
\begin{equation*}
\tilde{u}(x) := \frac{u_{z,r_j}(\rho x +\xi)}{\rho L},
\end{equation*}
where $L$ is chosen from \eqref{eq:uzrj-L}, so that we have 
\begin{equation}\label{eq:tilde u-lip}
\norm{\nabla \tilde{u}}_{L^\infty(B_1)}\leq \frac{1}{L}\norm{u_{z,r_{k_\delta}}}_{L^2(B_\rho(\xi))}\leq 1. 
\end{equation}
Observe that $\tilde{u}$ is a weak solution to 
\begin{equation*}
\ddiv(\tilde{A}(x,\tilde{u})\nabla\tilde{u})=0\quad\text{in }B_1, 
\end{equation*}
where $\tilde{A}(x,\tilde{u}) =\tilde{a}_+(x) H(\tilde{u}) + \tilde{a}_-(x)(1-H(\tilde{u}))$ with 
\begin{equation*}
\tilde{a}_\pm(x) := a_\pm(r_{k_\delta}\rho x + r_{k_\delta}\xi + z)\quad\text{in }B_1. 
\end{equation*}
Therefore, it follows from \eqref{eq:omega 0 rj1} that 
\begin{equation}\label{eq:tilde a-cont}
|\tilde{a}_\pm(x) - \tilde{a}_\pm(0)|\leq \omega_0r_{k_\delta}^\alpha\rho^\alpha|x|^\alpha\leq\delta\rho^\alpha|x|^\alpha\quad\text{in }B_1. 
\end{equation}
Moreover, in view of \eqref{eq:xi-rho}, we have that
\begin{equation}\label{eq:tilde u-flat}
\norm{\tilde{u}- P_{\frac{\beta}{L},\nu}}_{L^2(B_1)}\leq \frac{c_2\delta}{\rho L}.
\end{equation}

Owing to the inequalities in \eqref{eq:tilde u-lip}, \eqref{eq:tilde a-cont} and \eqref{eq:tilde u-flat}, $\tilde{u}$ and $\tilde{a}_\pm$ fall under the situation of Lemma \ref{lemma:approximation}; more specifically, we replace $\beta$ and $\e$ in \eqref{eq:assump-flat} with $\beta/L$ and respectively $\frac{c_2\delta}{\rho L}$, and make a further restriction on $\rho$ such that $\rho^{1+\alpha}\leq \frac{c_2\eta}{L}$ as well, where $\eta$ is the constant in \eqref{eq:assump-coeff}. Hence, we deduce from Corollary \ref{corollary:approximation} that there exists a nonzero vector $e^\xi\in\R^n$ such that 
\begin{equation*}
\Gamma(\tilde{u})\cap B_1 \subset\left\{x\in B_1: |x\cdot e^\xi|\leq \frac{c_3\e L}{\beta}|x|^{1+\alpha}\right\}.
\end{equation*}
In terms of $u_{z,r_{k_\delta}}$, we see that for any $\xi\in \Gamma(u_{z,r_{k_\delta}})\cap B_\rho$, 
\begin{equation*}
\Gamma(u_{z,r_{k_\delta}})\cap B_\rho(\xi)\subset \left\{x\in B_\rho(\xi): |(x-\xi)\cdot e^\xi|\leq \frac{c_4\delta}{\rho^{1+\alpha} \beta}|x-\xi|^{1+\alpha}\right\}.
\end{equation*}
Since the constant $c_4\delta/(\rho^{1+\alpha}\beta)$ is independent on $\xi$, we conclude that $\Gamma(u_{z,r_j})\cap B_\rho$, and thus, $\Gamma(u)\cap B_{r_{k_\delta}\rho}(z)$ is a $C^{1,\alpha}$ graph, proving the theorem. 
\end{proof}

Now we are left with proving Theorem \ref{theorem:Hn-1}. Before we begin, let us make an important observation on the nondegenerate part $\cN(u)$ of the free boundary, thanks to the monotonicity of the ACF formula. 

\begin{lemma}\label{lemma:N} $z\in\cN(u)$ if and only if there exist (small) positive numbers $\eta$ and $r_0$, both depending on $z$, such that 
\begin{equation*}
\frac{1}{r^{n/2}}\norm{u}_{L^2(B_r(z))}\geq \eta r\quad\text{for any }0<r\leq r_0.
\end{equation*}
\end{lemma}

\begin{proof} The `if' part is obvious by the definition of $\cN(u)$, whence we prove the `only if' part. 

Without loss of any generality, let us assume that $\norm{u}_{L^2(\Omega)}=1$ and suppose that the `only if' part fails to hold. Then there is $z\in\cN(u)$ which admits two sequences $r_j,t_j\searrow 0$ such that $r_j^{-n/2}\norm{u}_{L^2(B_{r_j/2}(z))}\geq \eta r_j$ for some $\eta>0$ while $t_j^{-n/2}\norm{u}_{L^2(B_{t_j}(z))}\leq \frac{1}{j}t_j$ for $j=1,2,\cdots$.

Define $u_j(x):=u(r_jx+z)/r_j$ and $v_j(x):=u(t_jx+z)/t_j$, in which case we have $\norm{u_j}_{L^2(B_1)}\geq \eta$. Without loss of any generality let us assume that $\sup_{B_1}u_j\geq c_1\eta$. As shown in the proof of Lemma \ref{lemma:gamma}, it follows that $\inf_{B_2}u_j\leq -c_2\eta$. The uniform Lipschitz regularity (Theorem \ref{theorem:opt reg}) implies that there are some balls $B_j^\pm\subset B_1$ such that $u_j\geq \frac{c_1\eta}{2}$ in $B_j^+$ and $u_j\leq -\frac{c_2\eta}{2}$ in $B_j^-$ with $\cL^n(B_j^\pm)\geq\delta>0$. Thus, the Poincar\'{e} inequality yields that
\begin{equation*}
c_{1,\delta}\eta\leq\norm{u_j^\pm}_{L^2(B_1)}\leq c_{2,\delta}\norm{\nabla u_j^\pm}_{L^2(B_1)}. 
\end{equation*}
By the monotonicity of the ACF formula (Proposition \ref{proposition:acf-monot}), 
\begin{equation*}
\Phi(0+,z,u^+,u^-)=\lim_{j\ra\infty}\Phi(r_j,z,u^+,u^-)=\lim_{j\ra\infty}\Phi(1,u_j^+,u_j^-)\geq c_\delta\eta^4;
\end{equation*}
see \eqref{eq:acf} for the definition of $\Phi$. 

On the other hand, $\norm{v_j}_{L^2(B_1)}\leq\frac{1}{j}$ implies that $\norm{\nabla v_j}_{W^{1,2}(B_{1/2})}\leq\frac{c_3}{j}$, and thus, $v_j\ra 0$ strongly in $W^{1,2}(B_{1/2})$. Then
\begin{equation*}
\Phi(0+,z,u^+,u^-)=\lim_{j\ra\infty}\Phi(t_j/2,z,u^+,u^-)=\lim_{j\ra\infty}\Phi(1/2,v_j^+,v_j^-)=0,
\end{equation*}
a contradiction. 
\end{proof}

\begin{remark}\label{remark:N} We may only assume Dini continuity on $a_+$ and $a_-$ to have Lemma \ref{lemma:N}, since its proof only involves the monotonicity of the ACF formula. 
\end{remark}

By Lemma \ref{lemma:N}, we may decompose the class $\cN(u)$ into a countable union of $\cN_{j,k}(u)$ for $j,k=1,2,\cdots$, where $z\in\cN_{j,k}(u)$ if 
\begin{equation*}
\frac{1}{r^{n/2}}\norm{u}_{L^2(B_r(z))}\geq\frac{1}{j}r\quad\text{for any }0<r\leq\frac{1}{k}.
\end{equation*}
Note that for any $z\in\cN_{j,k}(u)$, we have $\dist(z,\p\Omega)\geq\frac{1}{k}$. 

\begin{lemma}\label{lemma:N-hauss} For any pair $(j,k)$ of positive integers, 
\begin{equation*}
\cH^{n-1}(\cN_{j,k}(u))\leq C_{j,k}\norm{u}_{L^2(\Omega)}.
\end{equation*} 
\end{lemma}

\begin{proof} For simplicity let us assume that $\norm{u}_{L^2(\Omega)}=1$. Fix a pair $(j,k)$ of positive integers. Let us take a compact subset $D$ of $\Omega$ such that $\frac{1}{4k}<\dist(D,\p\Omega)<\frac{1}{2k}$. 

Select a smooth cutoff function $\phi$ on $\Omega$ such that $\phi\equiv 1$ in $D$, $\dist(\spt(\phi),\p\Omega)\geq\frac{1}{4k}$ and $|\nabla \phi|\leq Ck$. Given $\eta>0$, consider a function $\psi_\eta$ on $\R$ defined by $\psi_\eta(t) = 1$ if $t\geq \eta$, $\psi_\eta(t) = \eta^{-1}t$ if $0\leq t< \e$ and $\psi_\eta(t) = 0$ if $t<0$. By Lemma \ref{lemma:pos-msr}, we derive that 
\begin{equation}\label{eq:weak-N-hauss}
\int a_+\nabla u^+\nabla (\psi_\eta(u^+)\phi) \leq 0,
\end{equation}
by using $\psi_\eta(u^+)\phi$ as a test function. 

Noting that  
\begin{equation*}
\int a_+\nabla u^+\nabla(\phi_\eta(u^+)\phi) = \eta^{-1}\int_{\{0<u^+<\eta\}}a^+|\nabla u^+|^2\phi + \int \psi_\eta(u^+)a_+\nabla u^+\nabla \phi,
\end{equation*}
and using \eqref{eq:unif ellip}, we may derive from \eqref{eq:weak-N-hauss} that
\begin{equation*}
\begin{split}
\lambda\eta^{-1}\int_{\{0<u^+<\eta\}}|\nabla u^+|^2\phi &\leq \lambda^{-1}\int |\nabla u^+||\nabla\phi|.
\end{split}
\end{equation*}
Thus, from Theorem \ref{theorem:opt reg} and the construction of $\phi$, it follows that
\begin{equation*}
\int_{\{0<u^+<\eta\}\cap D}|\nabla u^+|^2 \leq ck^2\eta.
\end{equation*}
One may also notice that the above argument can also be applied to $u^-$, whence we conclude that
\begin{equation}\label{eq:grad u-N-hauss}
\int_{\{|u|<\eta\}\cap D}|\nabla u|^2 \leq 2ck^2\eta.
\end{equation}

Now let us take a countable open cover $\{B_i\}_{i=1}^\infty$ of $\cN_{j,k}(u)$, where for each $i$, $B_i$ is a ball of radius $\e$ with its center at $\cN_{j,k}(u)$. By the Vitali covering lemma, we may also choose this open cover to have finite overlapping times, whose finiteness, say $N$, depends only on the dimension $n$. Moreover, using the nondegeneracy of $u$ on $\cN_{j,k}(u)$ and the local Lipschitz regularity in Theorem \ref{theorem:opt reg}, it is not hard to observe that each $B_i$ contains subballs $B_i^{\pm}\subset\Omega^\pm(u)$ with radius $c_{j,k}\e$ where $u^\pm\geq\frac{\e}{2j}$. Owing to this fact, we are able to use the Poincar\'{e} inequality to deduce that 
\begin{equation}\label{eq:poin-N-hauss}
c_{j,k}\e^2\cL^n(B_i)\leq \int_{B_i}(u^\pm)^2\leq \tilde{c}_{j,k}\e^2\int_{B_i}|\nabla u^\pm|^2
\end{equation}

Now let us take $\e\leq\frac{1}{2k}$. Then $\dist(B_i,\p\Omega)\geq\frac{1}{2k}$, since we know that $\dist(\cN_{j,k}(u),\p\Omega)\geq\frac{1}{k}$. Because of Theorem \ref{theorem:opt reg}, we deduce that $B_i\subset\{|u|\leq ck\e\}\cap D$, where $D$ is the compact set chosen in the beginning of this proof. Hence, inserting $\eta=ck\e$ in \eqref{eq:grad u-N-hauss}, we derive from \eqref{eq:poin-N-hauss} that
\begin{equation*}
\sum_{i=1}^\infty\cL^n(B_i)\leq \hat{c}_{j,k}\sum_{i=1}^\infty\int_{B_i}|\nabla u|^2\leq\hat{c}_{j,k}N\int_{\{|u|\leq ck\e\}\cap D}|\nabla u|^2\leq 2c\hat{c}_{j,k}Nk^3\e.
\end{equation*}
Therefore, 
\begin{equation*}
\sum_{i=1}^\infty\diam(B_i)^{n-1}\leq C_{j,k},
\end{equation*}
and the conclusion of this lemma follows from the arbitrary choice of $\e\leq\frac{1}{2k}$. 
\end{proof}

\begin{remark}\label{remark:N-hauss} As in the proof of Lemma \ref{lemma:N}, that of Lemma \ref{lemma:N-hauss} only requires Dini continuity of $a_+$ and $a_-$, since the local Lipschitz regularity of $u$ in Theorem \ref{theorem:opt reg} played an essential role.  
\end{remark}

We are now ready to prove Theorem \ref{theorem:Hn-1}.

\begin{proof}[Proof of Theorem \ref{theorem:Hn-1}] Suppose towards a contradiction that $\mu(\cS(u))>0$. Then by Theorem \ref{theorem:C1a-fb}, we can find $z\in\cS(u)$ such that $B_r(z)\cap\Gamma(u)$ is a $C^{1,\alpha}$ graph for some $r>0$. In particular, from the proof of Theorem \ref{theorem:C1a-fb} and Corollary \ref{corollary:approximation}, we observe that for a sufficiently small $r_0$, we have 
\begin{equation*}
\frac{1}{\rho}u_{z,r_0}(\rho \cdot) \ra P_{\beta,e}\quad\text{strongly in $L^2(B_1)$ as $\rho\ra 0$},
\end{equation*}
for some $\beta>0$ and nonzero vector $e\in\R^n$. Therefore, 
\begin{equation*}
\lim_{\rho\ra 0}\frac{\norm{u}_{L^2(B_{\rho r_0}(z))}}{(\rho r_0)^{\frac{n}{2}+1}} = \frac{\norm{u}_{L^2(B_{r_0}(z))}}{r_0^{\frac{n}{2}+1}}\norm{P_{\beta,e}}_{L^2(B_1)}>0,
\end{equation*}
a contradiction against the assumption that $z\in\cS(u)$.

Thus, $\mu(\cS(u))=0$. Then it follows from $\cN(u)=\Gamma(u)\setminus\cS(u)$ that $\cN(u)$ has full $\mu$-measure. Due to Lemma \ref{lemma:N-hauss}, we know that $\cN(u)$ has $\sigma$-finite $(n-1)$-dimensional Hausdorff measure, and so does $\spt(\mu)$.

Finally, let us prove that $\cN(u)$ is relatively open in $\Gamma(u)$. Fix $z\in\cN(u)$. Then $Blo(u,z)$ contains a two plane solution, say $P_{\beta,\nu}$, for some $\beta>0$ and a unit vector $\nu\in\R^n$, due to the uniform Lipschitz regularity (Theorem \ref{theorem:opt reg}) and the nondegeneracy of $z$. Hence, we may go through the proof of Theorem \ref{theorem:C1a-fb} and conclude that $B_r(z)\cap \Gamma(u)$ is a $C^{1,\alpha}$ graph. Then we apply the Hopf lemma to \eqref{eq:main} in $B_r(z)\cap\Omega^\pm(u)$, and deduce that for all $\xi\in B_r(z)\cap\Gamma(u)$, $\frac{\p u^\pm}{\p \nu^\pm}(\xi)>0$, where $\nu^\pm$ is the inward unit normal to $\Omega^\pm(u)$ at $\xi$. Clearly, $\xi\in\cN(u)$, which proves that $B_r(z)\cap\Gamma(u)\subset \cN(u)$, as desired.  
\end{proof}

%
%

\section{Analysis on Matrix Coefficient Cases}\label{section:matrix}

Here we shall  extend our main results to (a special type of) matrix coefficient cases. Let $a_+$ and $a_-$ be  functions on $\R^n$ satisfying \eqref{eq:unif ellip} and \eqref{eq:cont}. Under this situation, consider a symmetric $(n\times n)$-matrix valued mapping $P$ on $\R^n$ satisfying
\begin{equation}\label{eq:ellip-matrix}
\lambda I\leq P(x)\leq \frac{1}{\lambda}I,\quad\forall x\in\R^n,
\end{equation}
and
\begin{equation}\label{eq:cont-matrix}
\norm{P(x)-P(y)}\leq \omega(|x-y|),\quad\forall x,y\in\R^n,
\end{equation}
where $\lambda$ and $\omega$ are the same quantities appearing in \eqref{eq:unif ellip} and respectively \eqref{eq:cont}. Define symmetric $(n\times n)$-matrix valued mappings $A_+$ and $A_-$ on $\R^n$ by 
\begin{equation}\label{eq:A_pm}
A_+(x) := a_+(x)P(x)\quad\text{and}\quad A_-(x):=a_-(x)P(x),
\end{equation}
and then a function $A:\R^n\times\R\ra R$ by
\begin{equation}\label{eq:matrix}
A(x,s):= A_+(x)H(s) + A_-(x)(1-H(s)),
\end{equation}
where $H$ is the Heaviside function. Set $\Omega$ to be a bounded domain in $\R^n$ and let $u$ be a weak solution to
\begin{equation}\tag{$P$}\label{eq:main-matrix}
\ddiv(A(x,u)\nabla u)= 0\quad\text{in }\Omega.
\end{equation}

\begin{remark}\label{remark:matrix} Alternatively, we may consider uniformly elliptic Dini continuous matrices $A_+$ and $A_-$ such that $A_+(x) = f(x)A_-(x)$ for a Dini continuous real valued $f$. 
\end{remark}

The existence of weak solutions to \eqref{eq:main-matrix} can be proved by the same argument in the proof of Proposition \ref{proposition:exist}, since there we only use the interior energy estimate for weak solutions to elliptic PDE with bounded measurable coefficients. 

It is noteworthy that the limiting equation of \eqref{eq:main-matrix} at a point $z\in\Omega$ is 
\begin{equation*} 
\ddiv((A_-(z) + (A_+(z) - A_-(z))H(v))\nabla v)=0,
\end{equation*}
whence the function $w$, defined by $w(x):=a_+(z)v^+(x) - a_-(z)v^-(x)$, solves 
\begin{equation*}
\ddiv(P(z)\nabla w)=0.
\end{equation*}
That is, $w$ is a harmonic function up to a bilinear transformation. Hence, the arguments throughout Section \ref{section:regularity} -- \ref{section:higher reg} are expected to go through with weak solutions to \eqref{eq:main-matrix} as well. 

The first result in concern with weak solutions to \eqref{eq:main-matrix} is the interior Lipschitz regularity of the associated weak solutions. 

\begin{theorem}\label{theorem:opt reg-matrix} Assume that $a_\pm$ and $P$ satisfy \eqref{eq:cont} and respectively \eqref{eq:cont-matrix} with a Dini continuous $\omega$ and let $u$ be a bounded weak solution to \eqref{eq:main-matrix} in $\Omega$. Then $u\in W_{loc}^{1,\infty}(\Omega)$ and for any $D\Subset\Omega$, 
\begin{equation*}
\norm{\nabla u}_{L^\infty(D)}\leq \frac{C\norm{u}_{L^2(\Omega)}}{d^{\frac{n}{2}+1}},
\end{equation*}
where $C$ is a constant depending only on $n$, $\lambda$ and $\omega$ and $d=\dist(D,\p\Omega)$.
\end{theorem}

\begin{proof} Note that Lemma \ref{lemma:Ca} can be extended to weak solutions of \eqref{eq:main-matrix} in an obvious way.  Moreover, by applying Theorem \ref{theorem:acf} instead of Proposition \ref{proposition:acf-est}, Lemma \ref{lemma:lip} continues to hold.
 Hence one may follow exactly the same argument as in Theorem \ref{theorem:opt reg} to derive Theorem \ref{theorem:opt reg-matrix}. We omit the details.
\end{proof}

Next we define the measure $\mu$ by $d\mu = \ddiv(A_+(x)\nabla u^+)$, and investigate the regularity of our free boundary. Let us hereafter follow Notation \ref{notation:u} and Definition \ref{definition:N-S} for the definition of $\Omega^+(u)$, $\Omega^-(u)$, $\Gamma(u)$, $\cN(u)$ and $\cS(u)$. 

\begin{theorem}\label{theorem:C1a-fb-matrix} Let $a_\pm$ and $P$ satisfy \eqref{eq:cont} and respectively \eqref{eq:cont-matrix} with an $\alpha$-H\"{o}lder continuous $\omega$, and $u$ be a weak solution to \eqref{eq:main-matrix} in $\Omega$. Then for $\mu$-a.e. $z\in\Gamma(u)$, there is $r>0$ such that $\Gamma(u)\cap B_r(z)$ is a $C^{1,\alpha}$ graph; here the radius $r$ and the $C^{1,\alpha}$ norm of the graph may depend on $u$, $z$ and the H\"{o}lder norm of $a_\pm$ and $P$.
\end{theorem}

\begin{proof} The arguments in Section \ref{section:prelim} and \ref{section:hausdorff} can be easily be generalized to weak solutions to \eqref{eq:main-matrix}. Observe that the function $v$, defined by $v(x):=a_+(0)u^+(x) - a_-(0)u^-(x)$, is a weak solution to 
\begin{equation*}
\ddiv(P(0)\nabla v) = \ddiv(\sigma(x)P(x)\nabla u),
\end{equation*}
where $\sigma$ is defined exactly the same as in the proof of Lemma \ref{lemma:approximation}. Under the assumptions on Theorem \ref{theorem:C1a-fb-matrix}, we have $\norm{\sigma\nabla u}_{L^\infty(B_r)}\leq\lambda^{-1}\omega_0 r^\alpha$ for some $\omega_0>0$ and $0<\alpha<1$. Thus, Lemma \ref{lemma:approximation} can also be extended to $u$. The rest of the proof follows similarly with that of Theorem \ref{theorem:C1a-fb}, and we skip the details.
\end{proof}

Finally we state our result on the measure theoretic regularity of $\Gamma(u)$ in the matrix coefficient case.

\begin{theorem}\label{theorem:Hn-1-matrix} Under the assumption of Theorem \ref{theorem:C1a-fb-matrix}, $\spt(\mu)$ has $\sigma$-finite $(n-1)$-dimensional Hausdorff measure. More specifically, $\cN(u)$ has $\sigma$-finite $(n-1)$-dimensional Hausdorff measure, while $\cS(u)$ has $\mu$-measure zero. In addition, the conclusion of Theorem \ref{theorem:C1a-fb-matrix} holds for every $z\in\cN(u)$, and $\cN(u)$ is relatively open in $\Gamma(u)$. 
\end{theorem}

\begin{proof} The proof of Lemma \ref{lemma:N} works with weak solutions to \eqref{eq:main-matrix} by invoking Theorem \ref{theorem:opt reg-matrix} and Theorem \ref{theorem:acf} instead of Theorem \ref{theorem:opt reg} and respectively Proposition \ref{proposition:acf-monot}. The proof of Lemma \ref{lemma:N-hauss} can be generalized to the matrix coefficient case in an obvious manner. Thus, one may follow the same arguments in the proof of Theorem \ref{theorem:Hn-1} to derive Theorem \ref{theorem:Hn-1-matrix}. We leave out the details to the reader.
\end{proof}

%
%

\appendix

\section{ACF Monotonicity Formula}\label{section:apndx}

For $r>0$, $z\in\R^n$ and $u\in W^{1,2}(B_r(z))$, set
\begin{equation*}
I(r,z,u):=\int_{B_r(z)}\frac{|\nabla u|^2}{{|x|}^{n-2}}dx,
\end{equation*}
and define the ACF functional $\Phi(r,u,v)$ by
\begin{equation}\label{eq:acf}
\Phi(r,z,u,v):=r^{-4}I(r,z,u)I(r,z,v)=\frac{1}{r^4}\int_{B_r(z)}\frac{|\nabla u|^2}{{|x|}^{n-2}}dx\int_{B_r(z)}\frac{|\nabla v|^2}{{|x|}^{n-2}}dx.
\end{equation}
For the notational convenience, let us abbreviate $I(r,0,u)$ and $\Phi(r,0,u,v)$ by $I(r,u)$ and respectively $\Phi(r,u,v)$. 

Let $\omega$ be a modulus of continuity and $\lambda$ a positive constant. Define $\cL(\lambda,\omega)$ by the class consisting of all elliptic operators $L=\ddiv(a(x)\nabla)$ such that $\lambda\leq a(x)\leq\lambda^{-1}$ and $|a(x)-a(y)|\leq\omega(|x-y|)$ for any $x,y\in B_1$. 

In what follows, any moduli of continuity $\omega$ is assumed to be Dini continuous; i.e., $\int_0^1\frac{\omega(r)}{r}dr<\infty$. Given such an $\omega$, define $\psi:(0,1)\ra[0,\infty)$ by
\begin{equation*}
\psi(r)=\omega(r)+\int_0^r\frac{\omega(\rho)}{\rho}d\rho+\left(\int_0^r\frac{\omega(\rho)}{\rho}\right)^2d\rho.
\end{equation*}

Here we follow the approach in \cite{CJK} and \cite{MP}.

\begin{lemma}\label{lemma:I-est} Let $L\in\cL(\lambda,\omega)$ and suppose that $u$ is a nonnegative weak solution for $Lu\geq 0$ in $B_1$. Then there are $c_0=c_0(n,\lambda)>0$ and $r_0=r_0(n,\lambda,\omega)$ such that 
\begin{equation*}
I(r,u)\leq (1+c_0\psi(r))\left(\frac{1}{r^{n-2}}\int_{\p B_r}u|\p_\nu u|+\frac{n-2}{2r^{n-1}}\int_{\p B_r}u^2\right)
\end{equation*}
for any $0<r\leq r_0$.
\end{lemma}

\begin{proof} Without loss of generality, we may assume that $a(0)=1$. Throughout the proof, $c$ will denote a dimensional constant which may vary from one line to another.

Write $L=\ddiv(a(x)\nabla)$ and $b(x)=1-a(x)$. Then since $\Delta u\geq\ddiv(b(x)\nabla u)$ in the sense of distribution, we have $\Delta(\frac{u^2}{2})\geq |\nabla u|^2+u\ddiv(b(x)\nabla u)$ in the same sense, which implies
\begin{equation}\label{eq:i-1}
I(r,u)\leq\frac{1}{2}\int_{B_r}\frac{1}{|x|^{n-2}}\Delta u^2-\int_{B_r}\frac{u}{|x|^{n-2}}\ddiv(b\nabla u).
\end{equation}
Since $\Delta|x|^{2-n}$ is a nonpositive measure in $\R^n$, 
\begin{equation}\label{eq:ori}
\begin{split}
\frac{1}{2}\int_{B_r}\frac{1}{|x|^{n-2}}\Delta u^2&\leq\frac{1}{r^{n-2}}\int_{\p B_r}u\p_\nu u+\frac{n-2}{2r^{n-1}}\int_{\p B_r}u^2,
\end{split}
\end{equation}
Hence, we only need to estimate the second term of the inequality \eqref{eq:i-1}. By means of integration by part, one can easily derive that
\begin{equation}\label{eq:pert}
\begin{split}
-\int_{B_r}\frac{u}{|x|^{n-2}}\ddiv(b\nabla u)&\leq\psi(r)\left(\frac{1}{r^{n-2}}\int_{\p B_r}u|\p_\nu u|+I(r,u)\right)\\
&\quad+(2-n)\int_{B_r}\frac{bu\nabla u\cdot x}{|x|^n}dx.
\end{split}
\end{equation}
Let us split the last term of the right hand side of \eqref{eq:pert} into two parts, namely,
\begin{equation}\label{eq:pert-r}
\begin{split}
\int_{B_r}\frac{bu\nabla u\cdot x}{|x|^n}dx&=\int_0^r\int_{\p B_\rho}\frac{bv_\rho\p_\nu u}{\rho^{n-1}}d\cH^{n-1} d\rho+\int_0^r\int_{\p B_\rho}\frac{m(\rho)b\p_\nu u}{\rho^{n-1}}d\cH^{n-1} d\rho,
\end{split}
\end{equation}
where $m(\rho):=\fint_{\p B_\rho}u$ and $v_\rho:=u-m(\rho)$. By the Poincar\'{e} inequality and the H\"{o}lder inequality, 
\begin{equation}\label{eq:pert-r1}
\begin{split}
\left|\int_0^r\int_{\p B_\rho}\frac{bv_\rho\p_\nu u}{\rho^{n-1}}d\cH^{n-1}d\rho\right|&\leq\psi(r)\int_0^r\frac{\kappa_n}{\rho^{n-2}}\left(\int_{\p B_\rho}|\nabla_\theta u|^2\right)^{\frac{1}{2}}\left(\int_{\p B_\rho}|\p_\nu u|^2\right)^{\frac{1}{2}}d\rho\\
&\leq\frac{\kappa_n\psi(r)}{2}I(r,u).
\end{split}
\end{equation}

On the other hand, the estimate for $\int_0^r\int_{\p B_\rho}\frac{m(\rho)b\p_\nu u}{\rho^{n-3}}d\cH^{n-1}d\rho$ can be done as follows. First using the equation $Lu\geq 0$, it is easy to see that 
\begin{equation*}
m(\rho)\leq m(\tau)-\frac{1}{n\omega_n}\int_{B_\tau\setminus B_\rho}\frac{b\nabla u\cdot x}{|x|^n}dx
\end{equation*}
whenever $0<\rho<\tau<1$. However, by the H\"{o}lder inequality,
\begin{equation*}
\begin{split}
\left|\int_{B_\tau\setminus B_\rho}\frac{b\nabla u\cdot x}{|x|^n}dx\right|&\leq\int_0^\tau\frac{\omega(\sigma)}{\sigma^{n-1}}\left(\int_{\p B_\sigma}|\nabla u|\right)d\sigma\leq\frac{(n\omega_n)^{1/2}}{2}\psi(r)I(\tau,u)^{1/2}.
\end{split}
\end{equation*}
Using these two inequalities, we obtain
\begin{equation}\label{eq:pert-r2}
\begin{split}
\left|\int_0^r\int_{\p B_\rho}\frac{m(\rho)b(x)\p_\nu u}{\rho^{n-1}}\right|&\leq \left(m(r)+\frac{\psi(r)}{2(n\omega_n)^{1/2}}I(r,u)^{1/2}\right)\frac{(n\omega_n)^{1/2}\psi(r)}{2}I(r,u)^{1/2}\\
&\leq\frac{1}{2}\psi(r)\left(\frac{1}{2r^{n-1}}\int_{\p B_r}u^2+I(r,u)\right),
\end{split}
\end{equation}
where in the last inequality we have used Young's inequality and the fact that
\begin{equation*}
m(r)^2=\left(\frac{1}{n\omega_nr^{n-1}}\int_{\p B_r}u\right)^2\leq\frac{1}{n\omega_nr^{n-1}}\int_{\p B_r}u^2.
\end{equation*}

Collecting the estimates \eqref{eq:pert-r1} and \eqref{eq:pert-r2}, we proceed in \eqref{eq:pert-r} as
\begin{equation*}
\begin{split}
\left|\int_{B_r}\frac{bu\nabla u\cdot x}{|x|^n}dx\right|\leq c\psi(r)\left(\frac{1}{2r^{n-1}}\int_{\p B_r}u^2+I(r,u)\right),
\end{split}
\end{equation*}
and hence, in \eqref{eq:pert}, we get
\begin{equation*}
\begin{split}
\left|\int_{B_r}\frac{u}{|x|^n}\ddiv(b\nabla u)\right|&\leq c\psi(r)\left(\frac{1}{r^{n-2}}\int_{\p B_r}u|\p_\nu u|+\frac{(n-2)}{2r^{n-1}}\int_{\p B_r}u^2+I(r,u)\right).
\end{split}
\end{equation*}
Inserting this inequality and \eqref{eq:ori} into \eqref{eq:i-1}, we arrive at
\begin{equation*}
(1-c\psi(r))I(r,u)\leq (1+c\psi(r))\left(\frac{1}{r^{n-2}}\int_{\p B_r}u|\p_\nu u|+\frac{n-2}{2r^{n-1}}\int_{\p B_r}u^2\right).
\end{equation*}
As a final step, we choose $r_0$ sufficiently small such that $c\psi(r)\leq\frac{1}{2}$ for any $r\in(0,r_0]$, and then set $c_0 = \frac{3}{2}c$, by which the proof is complete.
\end{proof}

Let us define a function $g$ on $[0,1)$ by 
\begin{equation*}
g(r) := \int_0^r \psi(\rho)d\rho.
\end{equation*}

\begin{proposition}\label{proposition:acf-monot} Let $L_\pm\in\cL(\lambda,\omega)$ and $u_\pm\in W^{1,2}(B_1)\cap C(\bar{B}_1)$ be such that
\begin{equation*}
u_\pm\geq 0, \quad u_+u_-=0\quad{and}\quad L_\pm u_\pm\geq 0\quad\text{in }B_1,
\end{equation*}
where the last equality is valid in the sense of distribution. Then the functional
\begin{equation*}
r\mapsto e^{\bar{c}g(r)}\Phi(r,u_+,u_-)
\end{equation*}
is monotone increasing in $(0,\bar{r}]$, for some positive $\bar{c}$ and $\bar{r}$, both depending only on $n$, $\lambda$ and $\omega$. 
\end{proposition}

\begin{proof} Since $u_\pm$ and $\omega$ are being fixed throughout the proof, we simplify our notations $\Phi(r,u_+,u_-)$ and $I(r,u_\pm)$ by $\Phi(r)$ and respectively $I_\pm(r)$. Denote by $\lambda_\pm(r)$ the principal eigenvalue of the spherical Laplacian in $\Sigma_\pm(r):=\{u_\pm>0\}\cap\p B_r$. The eigenvalue inequality reads
\begin{equation}\label{eq:eigen}
\lambda_\pm(r)\int_{\Sigma_\pm(r)}u_\pm^2\leq \int_{\Sigma_\pm(r)}|\nabla_\theta u|^2.
\end{equation}
Next write by $\beta_\pm(r)$ the characteristic constant for $\Sigma_\pm(r)$. Note that 
\begin{equation}\label{eq:char const}
\beta_\pm(r)(\beta_\pm(r)+n-2)=\lambda_\pm(r).
\end{equation}
We will use the Friedland-Hayman inequality,
\begin{equation}\label{eq:fh}
\beta_+(r)+\beta_-(r)\geq 2.
\end{equation}

Again for brevity, let $u$, $I(r)$, $\Sigma(r)$, $\alpha(r)$ and $\lambda(r)$ respectively stand for either $u_+$, $I_+(r)$, $\Sigma_+(r)$, $\beta_+(r)$ and $\lambda_+(r)$ or $u_-$, $I_-(r)$, $\Sigma_-(r)$, $\beta_-(r)$ and $\lambda_-(r)$. Applying Young's inequality and \eqref{eq:eigen} together with \eqref{eq:char const}, we derive that 
\begin{equation*}
\begin{split}
2r\int_{\Sigma(r)}u|\p_\nu u|+(n-2)\int_{\Sigma(r)}u^2&\leq(\alpha(r)+n-2)\int_{\Sigma(r)}u^2+\frac{r^2}{\alpha(r)}\int_{\Sigma(r)}|\p_\nu u|^2\\
&\leq\frac{r^2}{\alpha(r)}\int_{\Sigma(r)}|\nabla u|^2.
\end{split}
\end{equation*}
On the other hand, since $\nabla u\in L^2(B_1)$, $I(r)$ is absolutely continuous for $r\in(0,1)$, and hence, for a.e. $r\in(0,1)$, we have $I'(r)=\frac{1}{r^{n-2}}\int_{\Sigma(r)}|\nabla u|^2$. In view of Lemma \ref{lemma:I-est} and the inequality above, 
\begin{equation}\label{eq:i-i'}
I(r)\leq \frac{1+c_0\psi(r)}{2\alpha(r)}rI'(r),
\end{equation}
for a.e. $r\in(0,r_0)$, where $c_0$ and $r_0$ are chosen as in Lemma \ref{lemma:I-est}.

As $I_\pm(r)$ being absolutely continuous in $r$, we have, for a.e. $r\in(0,1]$, 
\begin{equation*}
\begin{split}
\frac{d}{dr}(e^{\bar{c}g(r)}\Phi(r))&=e^{\bar{c}g(r)}\Phi(r)\left[\frac{rI_+'(r)}{I_+(r)}+\frac{rI_-'(r)}{I_-(r)}-4+\bar{c}\psi(r)\right].
\end{split}
\end{equation*}
By means of \eqref{eq:fh} and \eqref{eq:i-i'}, we arrive at 
\begin{equation*}
\frac{d}{dr}(e^{\bar{c}g(r)}\Phi(r))\geq e^{\bar{c}g(r)}\Phi(r)\left[\frac{2(\beta_+(r)+\beta_-(r))}{1+c_0\psi(r)}-4+\bar{c}\psi(r)\right]\geq 0,
\end{equation*}
provided that we have choose $r_0$ such that $c_0\psi(r)<1$ for all $0<r\leq r_0$ and then $\bar{c}\geq 4c_0$. The proof is completed.
\end{proof}

As a result, we also get a useful estimate for the ACF functional.

\begin{proposition}\label{proposition:acf-est} Let $L_\pm$, $u_\pm$ and $r_0$ as in Proposition \ref{proposition:acf-monot}. Then
\begin{equation*}
\Phi(r,u_+,u_-)\leq C\norm{u_+}_{L^2(B_1)}^2\norm{u_-}_{L^2(B_1)}^2,
\end{equation*}
for any $0<r\leq r_0$, where $C$ depends only on $n$, $\lambda$ and $\omega$.
\end{proposition}

\begin{proof} It is not hard to see that $I(r_0,u_\pm)\leq N\norm{u_\pm}_{L^2(B_1)}^2$ for some $N=N(n,\lambda,\omega)$. With Proposition \ref{proposition:acf-monot} at hand, we obtain
\begin{equation*}
\Phi(r,u_+,u_-)\leq e^{\bar{c}(g(r_0)-g(r))}\Phi(r_0,u_+,u_-),
\end{equation*}
for $0<r\leq r_0$. The proof is now finished by taking $C=e^{\bar{c}g(r_0)}N^2/r_0^4$. 
\end{proof}

Let us make a further generalization to Proposition \ref{proposition:acf-monot} and \ref{proposition:acf-est}. 

\begin{theorem}\label{theorem:acf} The conclusions of Proposition \ref{proposition:acf-monot} and \ref{proposition:acf-est} remain to hold with $L_\pm = \ddiv(A_\pm(x)\nabla)$, where $A_+$ and $A_-$ are symmetric matrix valued functions on $B_1$ satisfying the following structure conditions:
\begin{enumerate}[(i)]
\item There is $0<\lambda<1$ such that 
\begin{equation*}
\lambda I\leq A_\pm(x) \leq\lambda^{-1}I\quad\text{in }B_1.
\end{equation*}
\item There is a Dini continuous modulus of continuity $\omega$ such that
\begin{equation*}
|A_\pm(x) - A_\pm(y)|\leq \omega(|x-y|)\quad\text{for any }x,y\in B_1.
\end{equation*} 
\item There exists a positive number $\kappa$ such that 
\begin{equation*}
A_+(0) = \kappa A_-(0).
\end{equation*}
\end{enumerate}
\end{theorem}

\begin{proof} One may notice that Lemma \ref{lemma:I-est} can be straightforwardly extended to the matrix coefficient case, provided that the associated matrix is symmetric uniformly elliptic and Dini continuous. Next in the proof of Proposition \ref{proposition:acf-monot}, the condition (iii) allows us to apply Lemma \ref{lemma:I-est} for $u_+$ and $u_-$ in the same coordinate, thus not affecting the Friedland-Hayman inequality. The rest of the proof follows in exactly the same way of those in Proposition \ref{proposition:acf-monot} and \ref{proposition:acf-est}, and hence, we omit the details.
\end{proof}

\end{document}